\documentclass[11pt,a4paper,reqno]{amsart}
\usepackage[left=3.5cm,right=3.5cm,top=2.2cm,bottom=2.2cm]{geometry}
\usepackage{bbm}
\usepackage{amsmath,amssymb,mathrsfs,stmaryrd}
\usepackage{amsthm}
\usepackage{mathtools,bm}
\usepackage{eufrak}
\usepackage{lmodern}
\usepackage{enumerate}
\usepackage[colorlinks,linkcolor=NavyBlue,citecolor=Gray,urlcolor=Periwinkle]{hyperref}
\usepackage[dvipsnames]{xcolor}
\usepackage[shortlabels] {enumitem}

\usepackage{lscape}
\usepackage{afterpage}
\theoremstyle{plain}
\newtheorem{theorem}{Theorem}[section]
\newtheorem{corollary}[theorem]{Corollary}
\newtheorem{lemma}[theorem]{Lemma}
\newtheorem{proposition}[theorem]{Proposition}
\theoremstyle{definition}
\newtheorem{definition}[theorem]{Definition}
\newtheorem{remark}[theorem]{Remark}

\numberwithin{equation}{section}
\newcommand{\kuo}[1]{\left( #1 \right)}
\newcommand{\zkuo}[1]{\left[ #1 \right]}
\newcommand{\dkuo}[1]{\left \{ #1 \right \}}
\newcommand{\e}{\mathbb{E}}

\newcommand{\abs}[1]{\left|#1\right|}
\newcommand{\norm}[1]{\left\|#1\right\|}
\newcommand{\brackt}[1]{\left\langle #1 \right\rangle}
\newcommand{\skpen}{u^{n,h}}
\newcommand{\contro}{Y^\varepsilon}
\newcommand{\pencontro}{Y^{n,\varepsilon}}

\newcommand{\penhsk}{Z^{n,\varepsilon}}

\begin{document}

\title{Large deviation principle for white noise SPDEs with oblique reflection}
\author{Hong Shaopeng}
\address{Department of Statistic and Data Science, School of Economics, Jinan University, Guangzhou, Guangdong, PR
China}
\email{hsp1999@stu2021.jnu.edu.cn}
\author{Liu Xiangdong}
\address{Department of Statistic and Data Science, School of Economics, Jinan University, Guangzhou, Guangdong, PR
China}
\email{tliuxd@jnu.edu.cn}


\begin{abstract}
In this paper, we consider Fredlin-Wentzell type large deviation principle (LDP) of multidimensional reflected stochastic partial differential equations in a convex domain, allowing for oblique direction of reflection. To prove the LDP, a sufficient condition for the weak convergence method and penalized method plays an important role.
\end{abstract}
\maketitle
\tableofcontents

\section{Introduction}

In recent years, stochastic partial differential equations (SPDEs) with reflection have become a highly researched topic in the field of stochastic analysis. Many scholars in this field have devoted significant research efforts to it. As an effective mathematical model, reflection SPDE has been widely applied in various fields, such as the description of limit order books and the characterization of multiphase systems, see e.g. \cite{hambly2020limit, muller2018stochastic} and their references.
As pioneers in this research area, \cite{DonatiMartin1993WhiteND} and \cite{Nualart1992WhiteND} respectively studied the existence of solutions for additive noise SPDEs and quasi-linear additive noise SPDEs using penalization methods. Meanwhile, \cite{Xu2009WhiteND} considered the well-posedness of parabolic SPDEs with reflection and proved the existence and uniqueness of solutions by utilizing comparison theorems. Compared to penalization methods, their approach provides a more concise proof and uses weak convergence methods \cite{budhiraja2000variational} to prove the large deviation principle of the SPDEs. \cite{Zhang2011WHITEND} used penalization methods to prove the existence and uniqueness of solutions for SPDEs with two reflecting walls. For reflection SPDEs in convex domains, since there are no comparison theorems for these types of equations, \cite{Zhang2011SystemsOS} obtained the uniqueness of solutions using finite-dimensional approximation methods and obtained the existence of solutions using penalization methods. Recently, \cite{Duan2019WhiteND} studied the well-posedness of oblique reflection SPDEs in convex domains and obtained the existence and uniqueness of solutions using penalization methods and a series of refined prior estimates.

The large deviation principle is an important tool for studying rare events and has always been a hot topic in the field of probability. \cite{freidlinRandomPerturbationsDynamical2012} first considered the large deviation principle for Markovian stochastic differential equations . In recent years, the weak convergence method proposed by \cite{budhiraja2019Analysis} has been well developed. This method uses variational inequalities to avoid the complex inequality estimates required when proving LDP.
Using the weak convergence method and the variational representation of Brownian motion, \cite{matoussi2021Large} proposed another sufficient condition for large deviations, which is more suitable for SPDEs with reflected than the original condition. This new sufficient and necessary condition has been widely applied to the proof of large deviations for reflecting SPDEs, see e.g. \cite{wang2022large, wang2022largeSD}. 

We note that while there has been a lot of research on the properties of reflected SPDEs, such as \cite{Dalang2004HittingPO,10.1214/16-AAP1186,Zhang2012LargeDF,XIE20195254, niu2019wang,zhang2010White} and their references, there has been relatively little study of oblique reflection SPDEs in convex domains. The aim of this paper is to establish a large deviation principle for oblique reflection SPDEs in convex domains using weak convergence methods.

The rest of the paper is organized as follows. In section \ref{sec:2}, we introduce the small noise white noise SPDEs with oblique reflection,  recall the weak convergence method of LDP and present a sufficient condition. In section \ref{sec:3}, we prove the existence and uniqueness of skeleton equation by penalized method. In section \ref{sec:4}, we show the LDP of small noise white noise SPDEs with oblique reflection.

We end this section with some notations. Let $H=L^2\left([0,1], \mathbb{R}^d\right)$ be the usual $L^2$-space with norm $\|\cdot\|_H$ and inner product $\langle\cdot, \cdot\rangle$. Denote by $\mathrm{V}$ the Sobolev space of order one, that is to say $\mathrm{V}$ is the completion of $C_0^{\infty}(0,1)$ under the norm $\|\cdot\|_V$ defined by $\|u\|_V^2=\int_0^1 \sum_{i=1}^d\left(\frac{\partial u_i}{\partial x}\right)^2 d x$. The corresponding inner product will be denoted by $\langle\cdot, \cdot\rangle_V$. Let $H^2$ be the Sobolev space of order two: $H^2=\left\{u \in V: \frac{\partial^2 u_i}{\partial x^2} \in H\right\}$. Let $\|\cdot\|_{H^2}$ be the norm of $H^2$ and $\langle\cdot, \cdot\rangle_d$ be the inner product in $\mathbb{R}^d$.
\section{Framwork}\label{sec:2}
\subsection{Small noise white noise SPDEs with oblique reflection}
Let $\mathcal{O}$ be a smooth, bounded convex domain in $\mathbb{R}^d$. Without loss of generality, we may assume $0 \in \mathcal{O}$. The following system of stochastic partial differential equations with oblique reflection:
\begin{equation}
\left\{\begin{aligned}
d u_i(t, x)=& \frac{\partial^2 u_i(t, x)}{\partial x^2} d t+b_i(u(t, x)) d t+\sum_{j=1}^m \sigma_{i j}(u(t, x)) d B_j(t) \\
&-\gamma_i(u(t, x)) d k_i(t, x), \quad x \in[0,1], \quad i=1, \ldots, d \\
u(0, \cdot)=&\left(u_1(0, \cdot), u_2(0, \cdot), \cdots, u_d(0, \cdot)\right)^T \in \overline{\mathcal{O}} \\
u(t, 0)=& u(t, 1)=0
\end{aligned}\right.,
\end{equation}
where $B=\left(B_1, B_2, \cdots, B_m\right)$ is an $m$-dimensional Brownian motion on a complete probability space $\left(\Omega,\left\{\mathcal{F}_t\right\}_{t \geq 0}, P\right)$, where $\left\{\mathcal{F}_t\right\}=\sigma(B(s): 0 \leq s \leq t)$. $u(0, \cdot)$ is a continuous function on $[0,1]$, which takes values in $\overline{\mathcal{O}}$ and vanishes at 0 and 1. Moreover, $\gamma$ is a smooth vector field on $\mathbb{R}^d$ standing for the direction of oblique reflection when the solution process $u$ hits the boundary $\partial \mathcal{O}$. $ k(t, x)=$ $\left(k_1(t, x), k_2(t, x), \cdots, k_d(t, x)\right)^T$, an $\mathbb{R}^d$-valued random measure, represents the size of force pushing the solution process back into the domain $\mathcal{O}$. The coefficients $b(y)=\left(b_1(y), b_2(y), \cdots, b_d(y)\right)^T$ and $\sigma(y)=$
$\left(\sigma_{i, j}(y), 1 \leq i \leq d, 1 \leq j \leq m\right)$ are measurable mappings from $\mathbb{R}^d$ into $\mathbb{R}^d$ and respectively from $\mathbb{R}^d$ into $\mathbb{R}^d \otimes \mathbb{R}^m$.

In this paper, We consider small noise SPDEs with oblique reflection \eqref{eq:sn_SPDE}.\begin{equation}\label{eq:sn_SPDE}
\left\{\begin{aligned}
d u^\varepsilon_i(t, x)=& \frac{\partial^2 u^\varepsilon_i(t, x)}{\partial x^2} d t+b_i(u^\varepsilon(t, x)) d t+\sqrt{\varepsilon}\sum_{j=1}^m \sigma_{i j}(u^\varepsilon(t, x)) d B_j(t) \\
&-\gamma_i(u^\varepsilon(t, x)) d k_i(t, x), \quad x \in[0,1], \quad i=1, \ldots, d \\
u^\varepsilon(0, \cdot)=&\left(u^\varepsilon_1(0, \cdot), u^\varepsilon_2(0, \cdot), \cdots, u^\varepsilon_d(0, \cdot)\right)^T \in \overline{\mathcal{O}} \\
u^\varepsilon(t, 0)=& u^\varepsilon(t, 1)=0
\end{aligned}\right..
\end{equation}

We need make some assumptions about domain $\mathcal{O}$, reflected direction and coefficients to make sure the the well-posedness of \eqref{eq:sn_SPDE}. 

\begin{enumerate}
	\item [\textbf{A.1}]Convex domain. $\mathcal{O}$ is a smooth, bounded convex domain in $\mathbb{R}^d$
	\item [\textbf{A.2}]Uniform interior cone. Let $x, z \in \mathbb{R}^d$ with $|z|=1$. For $b>0$ and $0 \leq a<1$, define the set
$$
C(x, z, b, a):=\left\{y \in \mathbb{R}^d:|y-x|<b,\langle y-x, z\rangle_d \geq a|y-x|\right\}
$$
Suppose that $\gamma \in C_b^2\left(\mathbb{R}^d\right)$ satisfies $|\gamma(x)|=1$ for $x \in \partial \mathcal{O}$. 
We have 
\begin{enumerate}[(A)]
\item  There exists $b>0$ and $0 \leq a<1$ such that for each $x \in \partial \mathcal{O}$,
$$
C(x, \gamma(x), b, a) \subset \mathcal{O}
$$
Let $n(x)$ be the unique outward unitary normal vector to $\partial \mathcal{O}$ at the point $x \in \partial \mathcal{O}$. If $\partial \mathcal{O} \in C^1$ and $\gamma$ is continuous, then condition (A) is equivalent to the following condition (B).
\item There exists $\rho>0$ such that for every $x \in \partial \mathcal{O}$,
$$
\langle\gamma(x), n(x)\rangle_d \geq \rho
$$
Condition (B) says that in the case of a smooth $\partial \mathcal{O}$, the uniform interior cone condition is equivalent to nontangentiality of $\partial \mathcal{O}$ to $\gamma$
\end{enumerate}

\item [\textbf{A.3}]Lipschitz condition. There exists a constant $C>0$ such that 
	$$
	|b(u)-b(v)| + |\sigma(u)-\sigma(v)|\leq C |u-v|
	$$
\end{enumerate}
\begin{remark}
	Since domain $\mathcal{O}$ is convex, then it satisfies uniform exterior sphere condition. That means there exists $r_0>0$ such that
	$$
B_{r_0}\left(x+r_0 n(x)\right) \subseteq \mathcal{O}^c, \forall x \in \partial \mathcal{O}
$$
where $n(x)$ is the unit outward normal and $B_\beta(x):=\left\{y \in \mathbb{R}^d:|y-x|<\beta\right\}, \beta>0$. Since $\mathcal{O}$ is convex, the unit outward normal equivalently there exists constant $C_0 \geq  1/(2r_0)$such that
$$
\langle x-y, n(x)\rangle_d+C_0|x-y|^2 \geq 0, \forall x \in \partial \mathcal{O}, \forall y \in \overline{\mathcal{O}}
$$
\end{remark}

From \cite{Lions1984StochasticDE} we have a Lemma on properties of the oblique vector field $\gamma$.
\begin{lemma}
Let $\gamma \in C_b^2\left(\mathbb{R}^d\right)$ satisfy $(2.3)$, then there exists a $d \times d$ symmetric matrix-valued function $\left(a_{i j}(x)\right)$ satisfying
$$
\begin{aligned}
&\left(a_{i j}(x)\right) \geq \theta I_d \quad \text { for some } \quad \theta>0, \quad a_{i j} \in C_b\left(\mathbb{R}^d\right) \\
&\sum_{j=1}^d a_{i j}(x) \gamma_j(x)=n_i(x) \quad \text { for } \quad 1 \leq i \leq d, \quad \forall x \in \partial \mathcal{O}.
\end{aligned}
$$
In particular, there exists $C_0 \geq 0$ such that
$$
C_0|x-y|^2+\sum_{i, j=1}^d a_{i j}(x)\left(x_i-y_i\right)\left(\gamma_j(x)\right) \geq 0, \quad \text { for all } x \in \partial \mathcal{O}, y \in \overline{\mathcal{O}}.
$$
In addition, if $\gamma \in C_b^1$ (resp. $\left.W^{1, \infty}\right)$, then $\left(a_{i j}\right) \in C_b^1\left(\right.$ resp. $\left.W^{1, \infty}\right)$. And if $\gamma \in C_b^2\left(\right.$ resp. $\left.W^{2, \infty}\right)$, then $\left(a_{i j}\right) \in C_b^2\left(\operatorname{resp} . W^{2, \infty}\right)$
\end{lemma}

We give the definition of the solution of \eqref{eq:sn_SPDE}.
\begin{definition}{}{}
A pair $\kuo{u^\varepsilon, \eta^\varepsilon}$ is said to be a solution of \eqref{eq:sn_SPDE} if 
\begin{enumerate}[(i)]
	\item $u^\varepsilon$ is a continuous random field on $\mathbb{R}_{+} \times[0,1] ; u(t, x)$ is $\mathcal{F}_t$ measurable and $u(t, x) \in \overline{\mathcal{O}}$ a.s.
	\item $\eta^\varepsilon$ is an $\mathbb{R}^d$-valued random vector on $\mathbb{R}_{+} \times[0,1]$ such that
	\begin{enumerate}
	\item $E\left[Var_{Q_T}(\eta^\varepsilon )\right]<\infty, \forall T \geq 0$, where $Var_{Q_T}(\eta^\varepsilon)$ denotes the total variation of $\eta$ on $Q_T=[0, T] \times[0,1]$.
	\item $\eta^\varepsilon$ is adapted in the sense that for any bounded measurable mapping $\psi$ :
$$
\int_0^t \int_0^1 \psi(s, x) \eta^\varepsilon(d s, d x) \text { is } \mathcal{F}_t \text { measurable. }
$$
	\end{enumerate}
	\item  $(u^\varepsilon,\eta^\varepsilon)$ solves the parabolic PDE with oblique reflection in the following sense: for any $t \in \mathbb{R}_+$, $\varphi\in C^2_0\kuo{[0,1];\mathbb{R}^d}$ with $\varphi(0)= \varphi(1) = 0$,
	$$
\begin{aligned}
&\langle u^\varepsilon(t), \varphi\rangle-\int_0^t\left\langle u^\varepsilon(s), \varphi^{\prime \prime}\right\rangle d s-\int_0^t\langle b(u^\varepsilon(s)), \varphi\rangle d s \\
&=\langle u(0), \varphi\rangle+\sqrt{\varepsilon}\sum_{k=1}^m \int_0^t\left\langle\sigma_k(u^\varepsilon(s)), \varphi\right\rangle d B_k(s)-\int_0^t \int_0^1 \varphi(x) \eta^\varepsilon(d s, d x) \quad a . s .,
\end{aligned}
$$
 \item for any $\phi\in C([0,T]\times[0,1];\bar{\mathcal{O}})$ 
	$$
	\int_0^T\int_0^1\brackt{u^\varepsilon(t,x)-\phi(t,x),\kuo{a_{ij}(u)}\eta^\varepsilon(dt,dx)}_d\geq 0
	$$

\end{enumerate}
\end{definition}

\subsection{Weak convergence approach: the abstract setting}
\begin{definition}[Large deviation]{}
 A family $\left\{X^{\varepsilon}\right\}_{\varepsilon>0}$ of $\mathcal{E}$-valued random variable is said to satisfy the large deviation principle on $\mathcal{E}$, with the good rate function $I$ and with the speed function $\lambda(\varepsilon)$ which is a sequence of positive numbers tending to $+\infty$ as $\varepsilon \rightarrow 0$, if the following conditions hold:
 \begin{enumerate}
 	\item for each $M<\infty$, the level set $\{x \in \mathcal{E}: I(x) \leq M\}$ is a compact subset of $E$;
 	\item for each closed subset $F$ of $\mathcal{E}, \limsup _{\varepsilon \rightarrow 0} \frac{1}{\lambda(\varepsilon)} \log \mathbb{P}\left(X^{\varepsilon} \in F\right) \leq-\inf _{x \in F} I(x)$;
 	\item for each open subset $G$ of $\mathcal{E}, \liminf _{\varepsilon \rightarrow 0} \frac{1}{\lambda(\varepsilon)} \log \mathbb{P}\left(X^{\varepsilon} \in G\right) \geq-\inf _{x \in G} I(x)$.
 \end{enumerate}
\end{definition}
 We recall here several results from \cite{budhiraja2019Analysis} and \cite{Matoussi2017LargeDP} which give an abstract framework of LDP and a sufficient condition.
 
 Let $\mathcal{H}$ be a Cameron-Martin space.
 $$
\mathcal{H}:=\left\{h:[0, T] \rightarrow \mathbb{R} ; h \text { is absolutely continuous and } \int_0^T|\dot{h}(s)|^2 d s<+\infty\right\}.
$$
 The space $\mathcal{H}$ is a Hilbert space with inner product $\left\langle h_1, h_2\right\rangle_{\mathcal{H}}:=\int_0^T\left\langle\dot{h}_1(s), \dot{h}_2(s)\right\rangle d s$. The Hilbert space $\mathcal{H}$ is endowed with the weak topology, i.e., for any $h_n, h \in \mathcal{H}, n \geq 1$, we say that $h_n$ converges to $h$ in the weak topology, if for any $g \in \mathcal{H}$,
$$
\left\langle h_n-h, g\right\rangle_{\mathcal{H}}=\int_0^T\left\langle\dot{h}_n(s)-\dot{h}(s), \dot{g}(s)\right\rangle d s \rightarrow 0, \quad \text { as } n \rightarrow \infty.
$$

Let $\mathcal{A}$ denote the class of real-valued $\left\{\mathcal{F}_t\right\}$-predictable processes $\phi$ belonging to $\mathcal{H}$ a.s. Let
$$
S_N:=\left\{h \in \mathcal{H} ; \int_0^T|\dot{h}(s)|^2 d s \leq N\right\} .
$$
$S_N$ is endowed with the weak topology induced from $\mathcal{H}$. Define
$$
\mathcal{A}_N:=\left\{\phi \in \mathcal{A} ; \phi(\omega) \in S_N, \mathbb{P} \text {-a.s. }\right\} .
$$

\begin{theorem}\label{thm:LDP_condition}
	For any $\varepsilon >0$, let $\mathcal{G}^\varepsilon:C([0,T];H)\rightarrow\mathcal{E}$ is a measurable mapping, $X^\varepsilon :=\mathcal{G}^\varepsilon\kuo{B(\cdot)}$. Assume there exists a measurable mapping $\mathcal{G}^0:C([0,T];H)\rightarrow \mathcal{E}$ such that 
	\begin{enumerate}
		\item [\textbf{(LDP1)}] For any $N<\infty$, $\dkuo{h^\varepsilon;\varepsilon>0}\in \mathcal{A}_N$ and $\delta>0$, we have 
		$$
		\lim_{\varepsilon \rightarrow 0}\mathbb{P}\kuo{\rho\kuo{Y^\varepsilon,Z^\varepsilon}>\delta} = 0
		$$
		where, $Y^\varepsilon = \mathcal{G}^\varepsilon\kuo{B(\cdot)+ \frac{1}{\sqrt{\varepsilon}}\int_0^{\cdot}\dot{h}^\varepsilon(s)ds}$, $Z^{\varepsilon}=\mathcal{G}^0\left(\int_0^{\cdot} \dot{h}^{\varepsilon}(s) d s\right)$, $\rho(\cdot,\cdot)$ is the metric of $\mathcal{E}$.
		\item [\textbf{(LDP2)}] For any $N<\infty$ and $\dkuo{h^\varepsilon;\varepsilon>0}\in S_N$, $h^\varepsilon \rightarrow h$ as $\varepsilon \rightarrow 0$, we have $\mathcal{G}^0\kuo{\int_{0}^\cdot \dot{h}^\varepsilon(s)ds}$ convergence to $\mathcal{G}^0\kuo{\int_0^{\cdot}\dot{h}(s)}ds$.
	\end{enumerate}
	Then, $\dkuo{X^\varepsilon}_{\varepsilon>0}$ satisfies LDP in $\mathcal{E}$, and rare function $I$ is 
	\begin{equation}
  I(g) = \inf_{\dkuo{h\in \mathcal{H},g = \mathcal{G}^0\kuo{\int_0^\cdot \dot{h}(s)ds}}}\dkuo{\frac{1}{2}\norm{h}^2_{\mathcal{H}}}.
\end{equation}
We denote $\inf \varnothing = \infty$
\end{theorem}
 \section{Well-posedness of the Skeleton equation}\label{sec:3}
For any $h \in \mathcal{H}$, consider the PDEs with oblique reflection \eqref{eq:skeleton}. 
\begin{equation}\label{eq:skeleton}
\left\{\begin{aligned}
d u^h_i(t, x)=& \frac{\partial^2 u^h_i(t, x)}{\partial x^2} d t+b_i(u^h(t, x)) d t+\sum_{j=1}^m \sigma_{i j}(u^h(t, x)) \dot{h}(t) dt \\
&-\gamma_i(u^h(t, x)) d k_i(t, x), \quad x \in[0,1], \quad i=1, \ldots, d \\
u^h(0, \cdot)=&\left(u^h_1(0, \cdot), u^h_2(0, \cdot), \cdots, u^h_d(0, \cdot)\right)^T \in \overline{\mathcal{O}} \\
u^h(t, 0)=& u^h(t, 1)=0
\end{aligned}\right..
\end{equation}
Let 
$$
\eta^h(t,x) = \int_0^t \gamma\kuo{u^h(s,x)}dk(s,x).
$$
The following is the definition of a solution to \eqref{eq:skeleton}.
\begin{definition}{}
A pair $(u^h,\eta^h)$ is said to be a solution of \eqref{eq:skeleton} if 
\begin{enumerate}[(i)]
	\item $u^h$ is a continuous field on $\mathbb{R}_+\times[0,1]$ and $u^h(t,x) \in \bar{\mathcal{O}}$
	\item $\eta^h$ is an $\mathbb{R}^d$ valued vector on $\mathbb{R}_+ \times [0,1]$ such that
	\begin{enumerate}
	\item $Var_{Q_T}(\eta^h)< \infty$, where $Var_{Q_T}(\eta^h)$ denotes the total variation of $\eta^h$ on $Q_T = [0,T]\times[0,1]$.
	\item For any bounded measurable mapping $\psi$:
	$$
	\int_0^t\int_0^1\psi(s,x)\eta^h(dt,dx)
	$$
	is measurable.
	\end{enumerate}
	\item $(u^h,\eta^h)$ solves the parabolic PDE with oblique reflection in the following sense: for any $t \in \mathbb{R}_+$, $\varphi\in C^2_0\kuo{[0,1];\mathbb{R}^d}$ with $\varphi(0)= \varphi(1) = 0$,
	\begin{equation}
		\begin{aligned}
			&\brackt{u^h(t),\varphi(t)}-\int_0^t\brackt{u^h(s),\varphi^{\prime\prime}(s)}ds - \int_0^t\brackt{b(u^h(s)),\varphi(s)}ds\\
			&=\brackt{u^h(0),\varphi}+\sum_{k=1}^m\int_0^t\brackt{\sigma_k(u^h(s))\dot{h}_k(s),\varphi}ds -\int_0^t\int_0^1\varphi(x)\eta^h(ds,dx).
		\end{aligned}
	\end{equation} 
	\item for any $\phi\in C([0,T]\times[0,1];\bar{\mathcal{O}})$ 
	$$
	\int_0^T\int_0^1\brackt{u^h(t,x)-\phi(t,x),\kuo{a_{ij}(u)}\eta^h(dt,dx)}_d\geq 0.
	$$
\end{enumerate}
\end{definition}

\begin{theorem}{}\label{thm:1}
Skeleton equation \eqref{eq:skeleton} exists unique solution $(u^h,\eta^h)$ in $C([0,T];H)\cap L^2([0,T];V)$.
\end{theorem}

For $y \in \mathbb{R}^d$, denote by $\pi(y)$ the projection of y onto the domain $\bar{\mathcal{O}}$. Since $\mathcal{O}$ is convex, $\pi$ is a contraction mapping, that means $\abs{\pi(x)-\pi(y)}\leq \abs{x-y}$, $\forall x \in \mathbb{R}^d$. Moreover, we may need the assumption:
\begin{equation}\label{eq:ass_gam}
  \exists \delta >0 \text{ such that } \brackt{\pi(x),\gamma(x)}_d \geq \delta, \quad x \in \mathbb{R}^d .
\end{equation}
 The penalized system of \eqref{eq:skeleton}.
 \begin{equation}\label{eq:penalized}
 \begin{aligned}
 	u^{n,h}(t,x) = & u(0,x) + \int_0^t \frac{\partial^2 u^{n,h}(s,x)}{\partial x^2}ds + \int_0^t b_i(u^{n,h}(s,x))ds + \sum_{j = 1}^m \int_0^t \sigma_{ij}(u^{n,h}(s,x))\dot{h}(s)ds \\
 	&- n\int_0^t \gamma(u^{n,h}(t,x))\abs{u^{n,h}(s,x) - \pi\kuo{u^{n,h}(s,x)}}ds.
 \end{aligned}
\end{equation}

We prepare a number of a priori estimates for \eqref{eq:penalized} in order to proof Theroem \ref{thm:1}
\begin{lemma}\label{lem:1}
	The following estimates hold.
	\begin{equation}\label{eq:lem_1_1}
  \sup_{n}\sup_{h \in \mathcal{S}_N} \sup_{0\leq t\leq T} \norm{u^{n,h}(t)}_H^4 < \infty
\end{equation}
	\begin{equation}\label{eq:lem_1_2}
\sup_n\sup_{h \in \mathcal{S}_N}\kuo{n\int_0^T dt \norm{u^{n,h}(t)}_H^2 \int_0^1 \left\langle  u^{n,h}(t,x),\gamma(u^{n,h}(t,x))\right\rangle_d\abs{u^{n,h}(t,x)-\pi(u^{n,h}(t,x))}dx}\leq L
\end{equation}
\end{lemma}
\begin{proof}
We have 
	\begin{equation}
  \begin{aligned}
\left\|u^{n, h}(t)\right\|_H^4&=\left\|u^{n, h}(0)\right\|_H^4+4 \int_0^t d s\left\|u^{n, h}(s)\right\|_H^2 \int_0^1\left\langle u^{n, h}(t, x), \frac{\partial^2 u^{n, h}(s, x)}{\partial x^2}\right\rangle_d d x \\
    &+ 4 \int_0^t \norm{\skpen(s)}^2_H \left\langle u^{n,h}(s) , b(\skpen(s)) \right\rangle_d ds \\
    &  +4\sum_{k=1}^m \int_0^t \norm{\skpen(s)}_H^2\left\langle\skpen,\sigma_k(\skpen(s))\dot{h}(s) \right\rangle ds\\
    & - 4n \int_0^t\norm{\skpen(s)}^2_H\left\langle\skpen(s),\gamma\kuo{\skpen(s)}\abs{\skpen(s,x) - \pi\kuo{\skpen(s,x)}} \right\rangle ds
    \end{aligned}
\end{equation}

Note that 
$$
4 \int_0^t d s\left\|\skpen(s)\right\|_H^2 \int_0^1\left\langle \skpen(s, x), \frac{\partial^2 \skpen(s, x)}{\partial x^2}\right\rangle_d d x=-4 \int_0^t\left\|\skpen(s)\right\|_H^2\left\|\frac{\partial \skpen(s)}{\partial x}\right\|_H^2 d s \leq 0
$$
$$
\begin{aligned}
&\left\langle \skpen(s), \gamma\left(\skpen(s)\right)\left|\skpen(s)-\pi\left(\skpen(s)\right)\right|\right\rangle \\
=& \int_0^1\left|\skpen(s, x)-\pi\left(\skpen(s, x)\right)\right|\left\langle u^n(s, x)-\pi\left(\skpen(s, x)\right), \gamma\left(\skpen(s, x)\right)\right\rangle_d d x \\
&+\int_0^1\left\langle\pi\left(\skpen(s, x)\right), \gamma\left(\skpen(s, x)\right)\right\rangle_d\left|\skpen(s, x)-\pi\left(\skpen(s, x)\right)\right| d x \\
\geq & \rho \int_0^1\left|\skpen(s, x)-\pi\left(\skpen(s, x)\right)\right|^2 d x+\delta \int_0^1\left|\skpen(s, x)-\pi\left(\skpen(s, x)\right)\right| d x \geq 0
\end{aligned}
$$

By Cauchy-Schwarz inequality and Young's inequality, we have
\begin{equation}
  \begin{aligned}
  	\sum_{k=1}^m\int_0^t &\norm{\skpen(s)}^2_H \left\langle u^{n,h}(s) , b(\skpen(s)) \right\rangle_d ds \\
  	&\leq \max _{0\leq s\leq T}\norm{\skpen(s)}^2_H \dkuo{\sum_{k=1}^m\int_0^t\brackt{\skpen(s,x),\sigma_k(\skpen(s,x))}\dot{h}_k(s)ds}\\
  	& \leq \max _{0\leq s\leq T}\norm{\skpen(s)}^2_H \int_0^t \int_0^1 u^{n, h}(s, x)\left(\sum_{j=1}^m\left|\sigma_j\left(u^{n, h}(s)\right)\right|^2\right)^{\frac{1}{2}}\left(\sum_{j=1}^m\left|\dot{h}_j(s)\right|^2\right)^{\frac{1}{2}} d x d s\\
  	&\leq \max _{0\leq s\leq T}\norm{\skpen(s)}^2_H C\left\{\int_0^t\left[\int_0^1 u^{n, h}(s, x)\left(1+\left|u^{n, h}(s, x)\right|^2\right)^{\frac{1}{2}} d x\right]^2 d s\right\}^{\frac{1}{2}}\left(\int_0^t|\dot{h}(s)|^2 d s\right)^{\frac{1}{2}}\\
  	& \leq C N\left(\sup _{0 \leq s \leq t}\left\|u^{n, h}(s)\right\|_H^2\right)\left(\int_0^t\left(1+\left\|u^{n, h}(s)\right\|_H^2\right) d s\right)^{\frac{1}{2}}\\
  	& \leq \frac{CN}{2}\left(\sup _{0 \leq s \leq t}\left\|u^{n, h}(s)\right\|_H^4\right) + \frac{1}{2}\left(\int_0^t\left(1+\left\|u^{n, h}(s)\right\|_H^2\right) d s\right)
  \end{aligned}
\end{equation}

We deduce that 
\begin{equation}
  \begin{aligned}
  	&\sup_{0\leq s\leq t}\norm{\skpen(s)}^4 + 4n \int_0^t\norm{\skpen(s)}^2_H\left\langle\skpen(s),\gamma\kuo{\skpen(s)}\abs{\skpen(s,x) - \pi\kuo{\skpen(s,x)}} \right\rangle ds \\
  	&\leq C \norm{\skpen(0)}^4_H + C \int_0^t\kuo{1+\norm{\skpen(s)}^4_H}ds 
  \end{aligned}
\end{equation}
which implies \eqref{eq:lem_1_1} and \eqref{eq:lem_1_2} by Gronwall's inequality.
\end{proof}
\begin{lemma}\label{lem:2}
 There exists a constant $M_T$ such that
$$
\sup_n\sup_{h\in \mathcal{S}_N} \left(n \int_0^T\left\|\skpen(t)-\pi\left(\skpen(t)\right)\right\|_{L^1([0,1])} d t\right)^2 \leq M_T, \quad T>0 
$$
\end{lemma}
\begin{proof}
	We have 
	\begin{equation}\label{eq:lem_2_1}
\begin{aligned}
	\norm{\skpen(t)}^2_H =& \norm{\skpen(0)}^2_H + 2 \int_0^t\brackt{\skpen(s),\frac{\partial^2 \skpen(s)}{\partial x^2}}ds+2\int_0^t\brackt{\skpen(s),b(\skpen(s))}ds\\
	&+ 2 \sum_{k=1}^m\int_0^t\brackt{\skpen(s),\sigma_k(\skpen(s))}h_k(s)ds \\
	&-2n\int_0^t\brackt{\skpen(s),\gamma(\skpen(s))\abs{\skpen(s)-\pi\kuo{\skpen(s)}}}ds
\end{aligned}
\end{equation}
Note that $\left\langle \skpen(s), \frac{\partial^2 \skpen(s)}{\partial x^2}\right\rangle=-\left\|\skpen(s)\right\|_V^2$ and
$$
\left\|\sigma_k\left(\skpen(s)\right)\right\|_H+\left\|b\left(\skpen(s)\right)\right\|_H \leq C\left(1+\left\|\skpen(s)\right\|_H\right) .
$$
In view of \eqref{eq:ass_gam}, we have 
\begin{equation}\label{eq:lem_2_2}
\begin{aligned}
& 2 n \int_0^t\left\langle \skpen(s), \gamma\left(\skpen(s)\right)\left|\skpen(s)-\pi\left(\skpen(s)\right)\right|\right\rangle d s \\
=& 2 n \int_0^t d s \int_0^1\left|\skpen(s, x)-\pi\left(\skpen(s, x)\right)\right|\left\langle \skpen(s, x)-\pi\left(\skpen(s, x)\right), \gamma\left(\skpen(s, x)\right)\right\rangle_d d x \\
+& 2 n \int_0^t d s \int_0^1\left\langle\pi\left(\skpen(s, x)\right), \gamma\left(\skpen(s, x)\right)\right\rangle_d\left|\skpen(s, x)-\pi\left(\skpen(s, x)\right)\right| d x \\
\geq & 2 \delta n \int_0^t\left\|\skpen(s)-\pi\left(\skpen(s)\right)\right\|_{L^1([0,1])} d s
\end{aligned}
\end{equation}

Combing \eqref{eq:lem_2_1}, \eqref{eq:lem_2_2} and lemma \ref{lem:1} , we have 
\begin{equation}
  \begin{aligned}
  	&\sup_n\sup_{h\in \mathcal{S}_N} \left(n \int_0^T\left\|\skpen(t)-\pi\left(\skpen(t)\right)\right\|_{L^1([0,1])} d t\right)^2 \\
  &	\leq C + C \sup_n \sup_{h\in \mathcal{S}_N}\sup_{0\leq t \leq T}\norm{\skpen(t)}_H^4\leq M_T
  \end{aligned}
\end{equation}

\end{proof}
Mimicking the proof of \cite{dupuis1993sdes}[Lemma 4.5], we can prove the following result.
\begin{lemma}\label{lem:3}
	It holds that 
	\begin{equation}
  \sup_n\sup_{h\in \mathcal{S}_N} \kuo{n^2\int_0^T\norm{\skpen(t) - \pi\kuo{\skpen(t)}}^2_Hdt} \leq C_T
\end{equation}
for some positive constant $C_T$
\end{lemma}
\begin{lemma}\label{lem:4}
	Assume that $u(0)\in V$. Then we have follow estimates
	\begin{equation}\label{eq:lem_4_1}
  \sup_n\sup_{h \in \mathcal{S}_N}\sup_{0\leq t \leq T} \norm{\skpen}^2_V < \infty 
\end{equation}
\begin{equation}\label{eq:lem_4_2}
\sup_n\sup_{h \in \mathcal{S}_N}\int_0^T \norm{\skpen(t)}^2_{H^2} < \infty
\end{equation}
\end{lemma}
\begin{proof}
	We have 
	\begin{equation}
  \begin{aligned}
  	\norm{\skpen(t)}_V^2 &= \norm{\skpen(0)}^2_V + 2 \int_0^t \brackt{\skpen(s),\frac{\partial^2 \skpen(s)}{\partial x^2}}_Vds + 2 \int_0^t \brackt{\skpen(s),b(\skpen(s)}_Vds\\
  	&+2\sum_{k=1}^m\int_0^t \brackt{\skpen(s),\sigma_k\kuo{\skpen(s)}}_V\dot{h}(s)ds \\
  	&-2n\int_0^t\brackt{\skpen(s),\gamma\kuo{\skpen(s)}\abs{\skpen(s) - \pi\kuo{\skpen(s)}}}_Vds
  \end{aligned}
\end{equation}

Using the integration by parts formula, we deduce that 
$$
\int_0^t \brackt{\skpen(s),\frac{\partial \skpen(s)}{\partial x^2}}_Vds = -\int_0^t\norm{\skpen(s)}^2_{H^2}ds \leq 0
$$
and 
\begin{equation}
  \begin{aligned}
  	&\brackt{\skpen(s),\gamma(\skpen(s))\abs{\skpen(s) - \pi\kuo{\skpen(s)}}}_V \\
  	&= -\int_0^1\brackt{\frac{\partial^2\skpen(s,x)}{\partial x^2},\gamma\kuo{\skpen(s) - \pi\kuo{\skpen(s)}}}_d dx
  \end{aligned}
\end{equation}

From the boundedness of $\gamma(x)$, we have 
\begin{equation}
  \begin{aligned}
\sup_{0\leq t \leq T}\norm{\skpen(t)}^2_V + 2 \int_0^T \norm{\skpen(s)}^2_{H^2}ds &\leq C \norm{\skpen(0)} + C \int_0^T \brackt{\skpen(s),b(\skpen(s)}_Vds\\
& +C\sum_{k=1}^m\int_0^T \brackt{\skpen(s),\sigma_k\kuo{\skpen(s)}}_V\dot{h}(s)ds\\
&+C n \int_0^T\norm{\skpen(s)}_{H^2}\abs{\skpen(s)- \pi\kuo{\skpen(s)}}_Hds
  \end{aligned}
\end{equation}

\begin{equation}
  \begin{aligned}
  	\sum_{k=1}^M\int_0^T \brackt{\skpen(s),\sigma_k(\skpen(s)}_V \dot{h}(s)ds &\leq C \kuo{\sum_{k=1}^M\int_0^T \brackt{\skpen(s),\sigma_k(\skpen(s)}_V^2ds}^\frac{1}{2}\kuo{\int_0^T\dot{h}^2(s)ds}^{\frac{1}{2}}\\
  	&\leq C_1 \kuo{\sum_{k=1}^M\int_0^T \brackt{\skpen(s),\sigma_k(\skpen(s)}_V^2ds}^\frac{1}{2}\\
  	&\leq \frac{1}{4}\sup_{0\leq t \leq T}\norm{\skpen(t)}^2_V + C_2 \sum_{k=1}^m \int_0^T\norm{\sigma_k\kuo{\skpen(s)}}^2_Vds
  \end{aligned}
\end{equation}

Using the following elementary inequality 
$$
0\leq ab\leq na^2 +\frac{1}{n}b^2
$$
we obtain that 
\begin{equation}
\begin{aligned}
	&Cn \int_0^T\norm{\skpen(s)}_{H^2}\abs{\skpen(s)- \pi\kuo{\skpen(s)}}_Hds\\
	&\leq C n \int_0^T\kuo{\frac{1}{Cn}\norm{\skpen(s)}^2_{H^2}+cn \norm{\skpen(s) - \pi \kuo{\skpen(s)}}_H^2}ds\\
	& = \int_0^T\norm{\skpen(s)}_{H^2}ds + Cn^2 \int_0^T\norm{\skpen(s)-\pi\kuo{\skpen(s)}}^2_Hds
\end{aligned}
\end{equation}

By Lemma \ref{lem:3}, the second term is bounded form above by $C_T$. Thus in view of the linear growth conditions of $b$ and $\sigma$, we conclude from the above displays that 
\begin{equation}
\begin{aligned}
	  \sup_{0\leq t \leq T}\norm{\skpen(t)}^2_V + 2 \int_0^T\norm{\skpen(s)}^2_{H^2}&\leq C\norm{\skpen(0)}^2_V + C_T \\
  &+ \int_0^T\norm{\skpen(s)}^2_{H^2}ds + C \int_0^T\kuo{1+\norm{\skpen(s)}^2_V}ds
\end{aligned}
\end{equation}
which implies \eqref{eq:lem_4_1} and \eqref{eq:lem_4_2}.
\end{proof}
By Lemma \ref{lem:4}, we have the following corollary.
\begin{corollary}\label{cor:1}
$$
	\sup_n\sup_{h\in \mathcal{S}_N} \sup_{0\leq t\leq T}\norm{\skpen(t)-\pi\kuo{\skpen(t)}}^2_V\leq \infty
$$
\end{corollary}
\begin{proof}
	Obversely, $\abs{\pi(x)-\pi(y)}\leq \abs{x-y}$ and the properties of the Sobolev space $V$, we deduce that $\brackt{\pi(u),\pi(u)}\leq \brackt{u,u}_V$. In addition, since 
	$$
	\norm{\skpen(t) -\pi\kuo{\skpen(t)}}_V\leq \norm{\skpen(t)}_V + \norm{\pi\kuo{\skpen(t)}}_V \leq 2 \norm{\skpen(t)}_V^2
	$$
then, by Lemma \ref{lem:4} we can get this corollary.
\end{proof}

\begin{lemma}\label{lem:5}
$$
	\lim\limits_{n\rightarrow \infty}\sup_{h\in \mathcal{S}_N}\sup_{0\leq t\leq T}\norm{\skpen(t) - \pi\kuo{\skpen(t)}}^2_H = 0
$$
\end{lemma}
\begin{proof}
	Let $q(z) = dist(z,\bar{\mathcal{O}})$, define $\psi(u) = \kuo{\int_0^1q\kuo{u(x)}ds}^2$, i.e. 
	$$
	\psi(u) = \norm{u - \pi(u)}^4_H\quad \text{for }u \in H 
	$$
	Then, the first Frechet derivative $\psi^\prime$ at $u$ is given as follow: for $h \in H$
	$$
	\psi^\prime(u)(h) = 4 \kuo{\int_0^1 q\kuo{u(x)}dx}\brackt{u-\pi(u),h}
	$$
	
	We have

\begin{equation}
\begin{aligned}
\psi(\skpen(t)) &=2 \int_0^t d s\left(\int_0^1 q\left(\skpen(s, x)\right) d x\right) \int_0^1 \sum_{i=1}^d \frac{\partial q}{\partial z_i}\left(\skpen(s, x)\right) \frac{\partial^2 \skpen_i(s, x)}{\partial x^2} d x \\
&+4 \sum_{k=1}^m \int_0^t\left(\int_0^1 q\left(\skpen(s, x)\right) d x\right)\left\langle \skpen(s)-\pi\left(\skpen(s)\right), \sigma_k\left(\skpen(s)\right)\right\rangle\dot{h}(s) d s \\
&+4 \int_0^t\left(\int_0^1 q\left(\skpen(s, x)\right) d x\right)\left\langle \skpen(s)-\pi\left(\skpen(s)\right), b\left(\skpen(s)\right)\right\rangle d s \\
&-4 n \int_0^t\left(\int_0^1 q\left(\skpen(s, x)\right) d x\right)\left\langle \skpen(s)-\pi\left(\skpen(s)\right), \gamma\left(\skpen(s)\right)\left|\skpen(s)-\pi\left(\skpen(s)\right)\right|\right\rangle d s \\
&+4 \sum_{k=1}^m \int_0^t\left\langle \skpen(s)-\pi\left(\skpen(s)\right), \sigma_k\left(\skpen(s)\right)\right\rangle^2 d s+\sum_{k=1}^m \int_0^t d s\left(\int_0^1 q\left(\skpen(s, x)\right) d x\right)\\
& = I_1 +I_2 +I_3 +I_4 + I_5
\end{aligned}
\end{equation}

In view of $\brackt{\gamma(u),n(u)}_d \geq \rho > 0$, it yields 
\begin{equation}
\begin{aligned}
&\left\langle \skpen(s)-\pi\left(\skpen(s)\right), \gamma\left(\skpen(s)\right)\left|\skpen(s)-\pi\left(\skpen(s)\right)\right|\right\rangle \\
&=\int_0^1\left|\skpen(s, x)-\pi\left(\skpen(s, x)\right)\right|\left\langle \skpen(s, x)-\pi\left(\skpen(s, x)\right), \gamma\left(\skpen(s, x)\right)\right\rangle_d d x \\
&\geq \rho\norm{\skpen(s)-\pi\left(\skpen(s)\right)}_H^2 \geq 0,
\end{aligned}
\end{equation}
which implies $I_4\leq 0$. Since $\mathcal{O}$ is convex and the function $q(z)$ is also convex on $\mathbb{R}^d\backslash\mathcal{O}$, the matrix$\left(\frac{\partial^2 q}{\partial z_i \partial z_j}\right)_{1 \leq i, j \leq d}$ is positive semi-definite on this domain. Then by the integration by parts formula, we have
$$
\begin{aligned}
I_1^n(t) &=-2 \int_0^t d s\left(\psi\left(\skpen(s)\right)\right)^{\frac{1}{2}} \int_0^1 \sum_{i, j=1}^d \frac{\partial^2 q}{\partial z_i \partial z_j}\left(\skpen(s, x)\right) \frac{\partial \skpen_i(s, x)}{\partial x} \frac{\partial \skpen_j(s, x)}{\partial x} d x \\
& \leq 0
\end{aligned}
$$

\begin{equation}
  \begin{aligned}
  	\sup_{h\in \mathcal{S}_N}\sup_{0\leq t \leq T}I_2 &\leq \sum_{k = 1}^m \int_0^T \psi\kuo{\skpen(s)}\brackt{\skpen(s)-\pi\left(\skpen(s)\right), \sigma_k\left(\skpen(s)\right)}^2ds\kuo{\int_0^T \dot{h}(s)^2ds}^{\frac{1}{2}}\\
  	& \leq C\sum_{k = 1}^m\int_0^T \psi\kuo{\skpen(s)}\brackt{\skpen(s)-\pi\left(\skpen(s)\right), \sigma_k\left(\skpen(s)\right)}^2ds\\
  	& \leq \frac{1}{4}\sup_{0\leq t\leq T}\psi\kuo{\skpen(t)}+ C_1 \kuo{\int_0^T \norm{\skpen(s)}^2_H\norm{\skpen(s)-\pi\kuo{\skpen(s)}}^2_Hds}
  \end{aligned}
\end{equation}
\begin{equation}
  \begin{aligned}
  	|I_3|&\leq \int_0^T\psi\kuo{\skpen(s)}^\frac{1}{2}\brackt{\skpen(s)-\pi\kuo{\skpen(s)},b\kuo{\skpen(s)}}ds\\
  	&\leq \frac{1}{4}\sup_{0\leq t \leq T}\psi\kuo{\skpen(t)} + C \kuo{\int_0^T \norm{\skpen(s)}_H\norm{\skpen(s)-\pi\kuo{\skpen(s)}}_Hds}^2\\
  	&\leq \frac{1}{4}\sup_{0\leq t \leq T}\psi\kuo{\skpen(t)}+ C_2\int_0^T\norm{\skpen(s)}^2_H\norm{\skpen(s) - \pi\kuo{\skpen(s)}}^2_Hds
  \end{aligned}
\end{equation}

And
\begin{equation}
  \begin{aligned}
  	&n\int_0^Tdt\norm{\skpen(t)}^2_H\int_0^1\brackt{\skpen(t,x),\gamma\kuo{\skpen(t,x)}}_d\abs{\skpen(t,x)-\pi\kuo{\skpen(t,x)}}dx\\
  	& = n \int_0^Tdt\norm{\skpen(t)}^2_H\int_0^1\brackt{\skpen(t,x)-\pi\kuo{\skpen(t,x)},\gamma\kuo{\skpen(t,x)}}_d\abs{\skpen(t,x)-\pi\kuo{\skpen(t,x)}}dx\\
  	&+n\int_0^Tdt\norm{\skpen(t)}^2_H\int_0^1\brackt{\pi\kuo{\skpen(t,x)},\gamma\kuo{\skpen(t,x)}}_d\abs{\skpen(t,x)-\pi\kuo{\skpen(t,x)}}dx\\
  	&\geq n \rho \int_0^T\norm{\skpen(t)}^2_H\norm{\skpen(t)-\pi\kuo{\skpen(t)}}^2_Hdt
  \end{aligned}
\end{equation}

Finally, we can get 
\begin{equation}
  \begin{aligned}
  	&\sup_{h\in \mathcal{S}_N}\sup_{0\leq t \leq T}\norm{\skpen(t)-\pi\kuo{\skpen(t)}}^2_H = \sup_{0\leq t \leq T}\psi\kuo{\skpen(t)}^\frac{1}{2}\leq \kuo{\sup_{0\leq t \leq T}\psi\kuo{\skpen(t)}}^{\frac{1}{2}}\\
  	&\leq C \kuo{\int_0^T\norm{\skpen(s)}^2_H\norm{\skpen(s)-\pi\kuo{\skpen(s)}}^2_Hds}^{\frac{1}{2}}\leq C \kuo{\frac{L}{n\rho}}^{\frac{1}{2}}\rightarrow 0 \quad \text{as }n \rightarrow \infty
  \end{aligned}
\end{equation}

\end{proof}
\begin{corollary}\label{cor:2}
$$
	\lim\limits_{n\rightarrow \infty} \sup_{h
	\in \mathcal{S}_N}\sup_{0\leq t \leq T} \norm{\skpen(t)- \pi\kuo{\skpen(t)}}^2_{L^\infty([0,1])} = 0
$$
\end{corollary}
\begin{proof}
	By the Sobolev embedding, for every $\varepsilon >0$, there exists a constant $C_\varepsilon$ such that
	\begin{equation}
  \begin{aligned}
   \sup_{h \in \mathcal{S}_N}\sup_{0\leq t \leq T} \norm{\skpen(t)- \pi\kuo{\skpen(t)}}^2_{L^\infty([0,1])}&\leq \varepsilon \sup_{h \in \mathcal{S}_N}\sup_{0\leq t \leq T} \norm{\skpen(t)- \pi\kuo{\skpen(t)}}^2_{V}\\
   &+ C_\varepsilon \sup_{h \in \mathcal{S}_N}\sup_{0\leq t \leq T} \norm{\skpen(t)- \pi\kuo{\skpen(t)}}^2_{H}
  \end{aligned}
\end{equation}
Letting $n\rightarrow \infty$, it follows Lemma \ref{lem:5} and Corollary \ref{cor:1} that 
$$
\lim\limits_{n\rightarrow \infty} \sup_{h
	\in \mathcal{S}_N}\sup_{0\leq t \leq T} \norm{\skpen(t)- \pi\kuo{\skpen(t)}}^2_{L^\infty([0,1])}\leq C\varepsilon
$$
Send $\varepsilon$ to 0 to prove the corollary.
\end{proof}
\begin{lemma}\label{lem:6}
For any $T>0$, $N>0$ and $n\geq m$, 
	\begin{equation}
\lim\limits_{n,m\rightarrow \infty} \sup_{h\in \mathcal{S}_N}\sup_{0\leq t\leq T}\norm{\skpen(t) - u^{m,h}(t)}^2_H +2 \sup_{h\in \mathcal{S}_N}\int_0^T\norm{\skpen(t)-u^{m,h}(t)}^2_Vdt = 0 
\end{equation}

\end{lemma}
\begin{proof}
	For $n\geq m$, we have 
	\begin{equation}
  \begin{aligned}
  	\norm{\skpen(t) - u^{m,h}(t)}^2_H &= 2\int_0^t \brackt{\skpen(s) - u^{m,h}(s),\frac{\partial^2\kuo{\skpen(s)-u^{m,h}(s)}}{\partial x^2}}ds\\
  	& +2\sum_{k=1}^m\int_0^t \brackt{\skpen(s) - u^{m,h}(s),\sigma_k\kuo{\skpen(s)}-\sigma_k\kuo{u^{m,h}(s)}}\dot{h}(s)ds\\
  	&+2 \int_0^t\brackt{\skpen(s)-u^{m,h}(s),b\kuo{\skpen(s)}-b\kuo{u^{m,h}(s)}}ds\\
  	& -2n\int_0^t\brackt{\skpen(s)-u^{m,h}(s),\gamma\kuo{\skpen(s)}\abs{\skpen(s) - \pi\kuo{\skpen(s)}}}ds\\
  	&+2m\int_0^t\brackt{\skpen(s)-u^{m,h}(s),\gamma\kuo{u^{m,h}(s)}\abs{u^{m,h}(s)-\pi\kuo{u^{m,h}(s)}}}ds\\
  	& = I_1^{m,n}+I_2^{m,n} + I_3^{m,n}+I_4^{m,n}+I_5^{m,n}
  \end{aligned}
\end{equation}

Using the integration by parts formula, we deduce that
\begin{equation}
  \brackt{\skpen(s)-u^{m,h}(s),\frac{\partial^2\kuo{\skpen(s)-u^{m,h}(s)}}{\partial x^2}} = - \norm{\skpen(s) - u^{m,h}(s)}^2_V\leq 0
\end{equation}

By Holder's inequality, we have 
\begin{equation}
  \begin{aligned}
  	&2n\int_0^t\brackt{\skpen(s) - \pi\kuo{u^{m,h}(s)},\gamma\kuo{\skpen(s)}\abs{\skpen(s) - \pi\kuo{\skpen(s)}}}ds\\
  	&\geq -Cn\int_0^t\norm{\skpen(s)-u^{m,h}(s)}_H \norm{\skpen(s) - \pi\kuo{\skpen(s)}}_Hds\\
  	& - Cn \int_0^t \norm{u^{m,h}(s)- \pi\kuo{u^{m,h}(s)}}_H\norm{\skpen(s)-\pi\kuo{\skpen(s)}}_Hds\\
  	&\geq -C\kuo{\int_0^t\norm{\skpen(s) - u^{m,h}(s)}^2_Hds}^{\frac{1}{2}}\kuo{n^2\int_0^t\norm{\skpen(s)-\pi\kuo{\skpen(s)}}^2_Hds}^{\frac{1}{2}}\\
  	&-C \sup_{0\leq s\leq t}\norm{u^{m,h}(s)-\pi\kuo{u^{m,h}}}_H\kuo{n \int_0^t\norm{\skpen(s)-\pi\kuo{\skpen(s)}}_Hds}
  \end{aligned}
\end{equation}

Consequently,
\begin{equation}
\begin{aligned}
I_4^{n, m}=&-2 n \int_0^t\left\langle \skpen(s)-\pi\left(u^{m,h}(s)\right), \gamma\left(\skpen(s)\right)\left|\skpen(s)-\pi\left(\skpen(s)\right)\right|\right\rangle d s \\
&+2 n \int_0^t\left\langle u^{m,h}(s)-\pi\left(u^{m,h}(s)\right), \gamma\left(\skpen(s)\right)\left|\skpen(s)-\pi\left(\skpen(s)\right)\right|\right\rangle d s \\
\leq & C\left(\int_0^t\left\|\skpen(s)-u^{m,h}(s)\right\|_H^2 d s\right)^{\frac{1}{2}}\left(n^2 \int_0^t\left\|\skpen(s)-\pi\left(\skpen(s)\right)\right\|_H^2 d s\right)^{\frac{1}{2}} \\
&+C \sup _{0 \leq s \leq t}\left\|u^{m,h}(s)-\pi\left(u^{m,h}(s)\right)\right\|_H \times\left(n \int_0^t\left\|\skpen(s)-\pi\left(\skpen(s)\right)\right\|_H d s\right)\\
&+\sup _{0 \leq s \leq t}\left\|u^{m,h}(s)-\pi\left(u^m(s)\right)\right\|_{L^{\infty}([0,1])} \times\left(2 n \int_0^t\left\|\skpen(s)-\pi\left(\skpen(s)\right)\right\|_{L^1([0,1])} d s\right)
\end{aligned}
\end{equation}

By a similar argument as above, the same estimate hods for $I^{m,n}_5$.

Combing the estimates for $I^{m,n}_4$ and $I^{m,n}_5$ and apply Lemma \ref{lem:2}, we have 
\begin{equation}
  \begin{aligned}
	&\sup_{0\leq t \leq T}\norm{\skpen(t)-u^{m,h}(t)}^2_H + 2 \int_0^T\norm{\skpen(t)- u^{m,h}(t)}^2_Vdt\\
	&\leq  
	\frac{1}{4}\sup_{0\leq t\leq T}\norm{\skpen(t) - u^{m,h}(t)}^2_H + C \int_0^T\norm{\skpen(t)-u^{m,h}(t)}^2_Hdt \\
	&+\left\{ C\kuo{C_T}^{\frac{1}{2}}\int_0^T\norm{\skpen(t)-u^{m,h}(t)}^2_Hdt+\sup_{0\leq t\leq T}\norm{\skpen(t)-\pi\kuo{u^{m,h}(t)}}_H^2\right. \\
	& \left. + \sup_{0\leq t\leq T}\norm{\skpen(t)-\pi\kuo{\skpen(t)}}^2_H\right\}\\
	&+ C\kuo{\sup_{0\leq t\leq T}\norm{u^{m,h}(t) - \pi\kuo{u^{m,h}(t)}}_{L^\infty([0,1])}}\times 2n\int_0^T\norm{\skpen(t)-\pi\kuo{\skpen(t)}}_{L^1([0,1])}dt \\
	&+C \kuo{\sup_{0\leq t\leq T}\norm{\skpen(t) - \pi\kuo{\skpen(t)}}_{L^\infty([0,1])}}\times 2n\int_0^T\norm{u^{m,h}(t)-\pi\kuo{u^{m,h}(t)}}_{L^1([0,1])}dt
\end{aligned}
\end{equation}

By Gronwall's inequality and Lemma \ref{lem:1} and \ref{lem:5}, we have 
\begin{equation}
  \begin{aligned}
  	&\sup_{h\in\mathcal{S}_N}\sup_{0\leq t\leq T}\norm{\skpen(t)-u^{m,h}(t)}^2_H + 2 \int_0^T\norm{\skpen(t)-u^{m,h}(t)}^2_Vdt\\
  	&\leq C\kuo{M_T}^\frac{1}{2}\sup_{0\leq t\leq T}\norm{u^{m,h}(t)-\pi\kuo{u^{m,h}(t)}}_{L^\infty([0,1])} \\
  	&+C\kuo{M_T}^\frac{1}{2}\sup_{0\leq t\leq T}\norm{\skpen(t)-\pi\kuo{\skpen(t)}}_{L^\infty([0,1])}
  \end{aligned}
\end{equation}

By corollary \ref{cor:2}, we have 
$$
\lim\limits_{n,m\rightarrow \infty} \sup_{h\in \mathcal{S}_N}\sup_{0\leq t\leq T}\norm{\skpen(t) - u^{m,h}(t)}^2_H +2 \sup_{h\in \mathcal{S}_N}\int_0^T\norm{\skpen(t)-u^{m,h}(t)}^2_Vdt = 0 
$$
\end{proof}
\begin{proof}[\textbf{Proof of Theorem \ref{thm:1}}] From Lemma \ref{lem:6}, there exists $u^h$ such that for any $T>0$, $u^h \in C\kuo{[0,T];H}\cap L^2\kuo{[0,T];V}$. We will show 
	$u^h$ is the solution of \eqref{eq:ass_gam}.
	
	From Lemma \ref{lem:5}, it follows that 
	$$
	\sup_{h\in \mathcal{S}_N}\sup_{0\leq t\leq T}\norm{u^h(t) - \pi(u^h(t))}^2_H \leq \lim\limits_{n\rightarrow \infty}\sup_{h\in \mathcal{S}_N}\sup_{0\leq t \leq T}\norm{\skpen(t)-\pi\kuo{\skpen(t)}}^2_H =0
	$$
	This means that for any $t >0$, we have $dist(u^h(t,x),\bar{\mathcal{O}}) = 0$ for almost all $x\in [0,1]$. Letting 
	$$
	\eta^{n,h}(dt,dx) = n\gamma\kuo{\skpen(t,x)}\abs{\skpen - \pi \kuo{\skpen(t,x)}}dtdx
	$$
	From every $T>0$, by Lemma \ref{lem:1}, we have 
	\begin{equation}\label{eq:bound_eta_h}
  \sup_n\sup_{h\in \mathcal{S}_N} Var\kuo{\eta^{n,h}}\kuo{[0,T]\times [0,1]}^2\leq \sup_n\sup_{h\in \mathcal{S}_N}\kuo{n\int_0^T\norm{\skpen(t) - \pi\kuo{\skpen(t)}}_{L^1([0,1])}dt}^2\leq \infty
\end{equation}
where $Var\kuo{\eta^{n,h}}\kuo{[0,T]\times [0,1]}$ denotes the total variation of $\eta^{n,h}$ on $Q_T=[0,T]\times [0,1]$. Let $\mathcal{M}\kuo{Q_T}$ be the Banach space of measures on $Q_T$ with the norm of total variation. It follows from \eqref{eq:bound_eta_h} that $\dkuo{\eta^{h,n}(dt,dx)}$ is bounded in $L^2\kuo{\mathcal{M}\kuo{Q_T}}$. Since $\mathcal{M}\kuo{Q_T}$ can be identified with the dual of $C\kuo{Q_T}$, $\eta^{h,n}$ is converges to an element $\eta^h \in L^2\kuo{\Omega,\mathcal{M}\kuo{Q_T}}$ with respect to the weak-*-topology. From \eqref{eq:bound_eta_h}, we also have $Var\kuo{\eta^h\kuo{[0,T]\times[0,1]}}<\infty$. 

Taking any $\varphi \in C^2_0\kuo{(0,\infty)\times [0,1];\mathbb{R}^d}$, by chain rule we have 
\begin{equation}
  \begin{aligned}
  	&\brackt{\skpen(t),\varphi(t)} - \int_0^t \brackt{\skpen(s),\frac{\partial \varphi(s)}{\partial x}}ds - \int_0^t\brackt{\skpen(s),\varphi^{\prime\prime}(s)}ds\\
  	&=\brackt{u(0),\varphi(0)}+\sum_{k=1}^m\int_0^t\brackt{\sigma_k(\skpen(s)),\varphi(s)}\dot{h}(s)ds \\
  	&+\int_0^t\brackt{b\kuo{\skpen(s),\varphi(s)}}ds-\int_0^t\int_0^1\varphi(s,x)\eta^{n,h}(ds,dx)
  \end{aligned}
\end{equation}
All the terms on the right hand side of the above identity converges. As $n\rightarrow \infty$, we have 
\begin{equation}
  \begin{aligned}
  	&\brackt{u^h(t),\varphi(t)} - \int_0^t \brackt{u^h(s),\frac{\partial \varphi(s)}{\partial x}}ds - \int_0^t\brackt{u^h(s),\varphi^{\prime\prime}(s)}ds\\
  	&=\brackt{u(0),\varphi(0)}+\sum_{k=1}^m\int_0^t\brackt{\sigma_k(u^h(s)),\varphi(s)}\dot{h}(s)ds \\
  	&+\int_0^t\brackt{b\kuo{u^h(s),\varphi(s)}}ds-\int_0^t\int_0^1\varphi(s,x)\eta^{h}(ds,dx)
  \end{aligned}
\end{equation}

For any $\phi \in C([0, T] \times[0,1] ;\overline{\mathcal{O}})$, we obtain that 
$$\left\langle \skpen(t, x)-\phi(t, x), \skpen(t, x)-\pi\left(\skpen(t, x)\right)\right\rangle_d \geq 0.$$
 Since $\sum_{j=1}^d a_{i j}(x) \gamma_j(x)=n_i(x)$ for any $x \in \partial \mathcal{O}$, we have 
$$\left(a_{i j}\left(\skpen\right)\right) \eta^{n,h}(d t, d x)=\left(a_{i j}\left(\skpen\right)\right) \skpen(t, x)-\pi\left(\skpen(t, x)\right) d t d x $$
and then
$$
\left\langle \skpen(t, x)-\phi(t, x),\left(a_{i j}\left(\skpen\right)\right) \eta^{n,h}(d t, d x)\right\rangle_d \geq 0 .
$$
Letting $n \rightarrow \infty$, it yields
$$
\int_0^T \int_0^1\left\langle u^h(t, x)-\phi(t, x),\left(a_{i j}(u)\right) \eta^h(d t, d x)\right\rangle_d \geq 0,
$$
by the strong convergence of $\skpen$ in $L^2\left(\Omega, C\left(Q_T\right)\right)$ combined with the Sobolev embedding. Therefore, we conclude from above displays that $(u^h, \eta^h)$ is a solution to \eqref{eq:skeleton}.

Let $\kuo{v^h,\eta^h_2(dt,dx)}$ be another solution to the skeleton equation \eqref{eq:skeleton} such that $\sup_{0\leq t\leq T}\norm{v^h(t)}^2_H\leq\infty$ for any $T>0$. There exists a function $\Phi \in C^2_b\kuo{\mathbb{R}^d}$ such that
\begin{equation}
\exists \alpha>0, \quad \forall u \in \partial \mathcal{O}, \quad\langle\nabla \Phi(u), \gamma(u)\rangle_d \leq-\alpha C_0 \leq 0.
\end{equation}

Define
\vspace{-0.3cm}
$$
\phi(u^h(t)):=\int_0^1 \Phi(u^h(t, x)) d x
$$
By chain role, we have 
\begin{equation}
  \begin{aligned}
  	\phi\kuo{u^h(t)} =& \phi\kuo{u^h(0)}+\int_0^t\brackt{\nabla \Phi\kuo{u^h(s)},\frac{\partial^2u^h(s)}{\partial x^2}}ds+ \int_0^t\brackt{\nabla\Phi(u^h(s),b\kuo{u^h(s)}}ds\\
  	&+ \sum_{k=1}^m\int_0^t\brackt{\nabla \Phi\kuo{u^h(s)},\sigma_k\kuo{u^h(s)}}\dot{h}(s)ds - \int_0^t\brackt{\nabla\Phi\kuo{u^h(s)},d\eta_1(s)}
  \end{aligned}
\end{equation}
and
\begin{equation}
  \begin{aligned}
  	de^{\dkuo{-\lambda \kuo{\phi\kuo{u^h(t)}+\phi\kuo{v^h(t)}}}}&=-\lambda e^{\dkuo{-\lambda \kuo{\phi\kuo{u^h(t)}+\phi\kuo{v^h(t)}}}} d\kuo{\phi\kuo{u^h(t)}+\phi\kuo{v^h(t)}}\\
  	& = -\lambda e^{\dkuo{-\lambda \kuo{\phi\kuo{u^h(t)}+\phi\kuo{v^h(t)}}}} \left\{\left[\brackt{\nabla \Phi\kuo{u^h(t)},\frac{\partial^2u^h(t)}{\partial x^2}}\right.\right.\\
  	&\left.+\brackt{\nabla \Phi\kuo{v^h(t)},\frac{\partial^2v^h(t)}{\partial x^2}} \right]dt +\sum_{k=1}^m \left[\brackt{\nabla \Phi\kuo{u^h(t)},\sigma_k\kuo{u^h(t)}}\dot{h}(s) \right.\\
  	&+\left. \brackt{\nabla \Phi\kuo{v^h(t)},\sigma_k\kuo{v^h(t)}}\dot{h}(s)\right]dt+ \left[\brackt{\nabla \Phi\kuo{u^h(t)},b\kuo{u^h(t)}}\right.\\
  	&+\left.\brackt{\nabla \Phi\kuo{u^h(t)},b\kuo{u^h(t)}}\right]dt-\brackt{\nabla\Phi\kuo{u^h(t)},d\eta_1^h(t)}\\
  	&\left.-\brackt{\nabla\Phi\kuo{v^h(t)},d\eta_2^h(t)}\right\}
  \end{aligned}
\end{equation}

Define 
$$
\varphi(t):=\int_0^1\left[a_{i j}(u(t, x))+a_{i j}(v(t, x))\right]\left(u_i(t, x)-v_i(t, x)\right)\left(u_j(t, x)-v_j(t, x)\right) d x
$$
By chain rule we have 
\begin{equation}
  \begin{aligned}
  	d\varphi(t) &= 2\int_0^1\zkuo{a_{ij}(u^h(t,x))+a_{ij}(v^h(t,x))}\kuo{u^h_i-v^h_i}(t,x)\times\left\{\frac{\partial^2\kuo{u^h_i-v^h_i}}{\partial x^2}(t,x)dt\right.\\
  	&+\kuo{b_j\kuo{u^h}-b_j\kuo{v^h}}(t,x)dt + \sum_{k=1}^m\kuo{\sigma_{jk}\kuo{u^h}-\sigma_{jk}(v^h)}(t,x)\dot{h}(t)dt\\
  	&\left.-\gamma_j\kuo{u^h(t,x)}dK_1(t,x) + \gamma_j\kuo{v^h(t,x)}dK_2(t,x)\right\}dx\\
  	&+\int_0^1\left[a_{i j}^{\prime}(u^h(t, x)) d u^h_i(t, x)+a_{i j}^{\prime}(v^h(t, x)) d v^h_i(t, x)\right]\left(u^h_i-v^h_i\right)\left(u^h_j-v^h_j\right)(t, x) d x
  \end{aligned}
\end{equation}

 $$
\frac{1}{\alpha}\langle\nabla \Phi(u), \gamma(u)\rangle_d|u-v|^2-\sum_{i, j=1}^d a_{i j}(u)\left(u_i-v_i\right) \gamma_j(u) \leq 0
$$
$$
\frac{1}{\alpha}\langle\nabla \Phi(v), \gamma(v)\rangle_d|u-v|^2-\sum_{i, j=1}^d a_{i j}(v)\left(v_i-u_i\right) \gamma_j(v) \leq 0.
$$
We also have 
\begin{equation}
\begin{aligned}
& \int_0^1\left[a_{i j}(u^h(t, x))+a_{i j}(v^h(t, x))\right]\left(u^h_i-v^h_i\right)(t, x) \frac{\partial^2\left(u^h_j(t, x)-v^h_j(t, x)\right)}{\partial x^2} d x \\
&\leq  C \int_0^1\left(u^h_i-v^h_i\right)(t, x) \frac{\partial^2\left(u^h_i(t, x)-v^h_i(t, x)\right)}{\partial x^2} d x \\
&=-C \int_0^1\left(\frac{\partial\left(u^h_i(t, x)-v^h_i(t, x)\right)}{\partial x}\right)^2 d x \leq 0 .
\end{aligned}
\end{equation}
\begin{equation}
\begin{aligned}
& e^{-\lambda\left(\phi\left(u^h_t\right)+\phi\left(v^h_t\right)\right)} \varphi(t) \\
&\leq  C_1  \int_0^t\norm{u^h(s)-v^h(s)}_H^2 e^{-\lambda\left(\phi\left(u^h_t\right)+\phi\left(v^h_t\right)\right)} d s \\
&+ C_2  \int_0^t e^{-\lambda\left(\phi\left(u^h_t\right)+\phi\left(v^h_t\right)\right)} \int_0^1|u^h(s, x)-v^h(s, x)|^2\left(d K_1(s, x)+d K_2(s, x)\right) \\
 & -2\int_0^t e^{-\lambda\left(\phi\left(u^h_t\right)+\phi\left(v^h_t\right)\right)} \int_0^1\left[a_{i j}(u^h(s, x))+a_{i j}(v^h(s, x))\right] \\
& \times\left(u^h_i(s, x)-v^h_i(s, x)\right)\left(\gamma_j(u^h) d K_1(s, x)-\gamma_j(v^h) d K_2(s, x)\right) \\
&-\alpha C_0 \lambda \theta  \int_0^t e^{-\lambda\left(\phi\left(u^h_t\right)+\phi\left(v^h_t\right)\right)} \int_0^1|u^h(s, x)-v^h(s, x)|^2\left(d K_1(s, x)+d K_2(s, x)\right) .
\end{aligned}
\end{equation}

Taking $\lambda = \frac{C_2 + 2C_0}{\alpha C_0\theta}$ we deduce that 
$$
e^{-\lambda\left(\phi\left(u^h_t\right)+\phi\left(v^h_t\right)\right)} \varphi(t) \leq C^{\prime} \left[\int_0^t\norm{u^h(s)-v^h(s)}_H^2 e^{-\lambda\left(\phi\left(u^h_s\right)+\phi\left(v^h_s\right)\right)} d s\right]
$$
which implies that 
$$
\sup_{t\leq T}\norm{u^h(t)- v^h(t)}_H^2\leq C\int_0^T\sup_{t\leq T}\norm{u^h(t)-v^h(t)}^2_Hdt
$$
\vspace{-0.3cm}
Therefore by Gronwall's inequality, we can show that $u^h(t)=v^h(t)$
\end{proof}

In order to prove proposition \ref{pro:1}, we need this lemma.
\begin{lemma}\label{lem:7}
	For $N>1$, $h_{r}$ weak convergence to $h$ as $r \rightarrow 0$, then for any $T>0$, $n\geq 1$
\begin{equation}
\lim _{r \rightarrow 0}\left(\sup _{0 \leq t \leq T}\left\|u^{n, h_{r}}(t)-u^{n, h}(t)\right\|_H^2+\int_0^T\left\|u^{n, h_{r}}(t)-u^{n, h}(t)\right\|_V^2 d t\right)=0 .
\end{equation}
\end{lemma}
\begin{proof}
\begin{equation}
  \begin{aligned}
  	\norm{u^{n,h_r}(t) - u^{n,h}(t)}^2_H & = 2\int_0^t\brackt{u^{n,h_r}(s) - u^{n,h}(s),\frac{\partial^2\kuo{u^{n,h_r}(s)-u^{n,h}(s)}}{\partial x^2}}ds\\
  	&+2\sum_{k=1}^m\int_0^t \brackt{u^{n,h_r}(s) - u^{n,h}(s),\sigma_k\kuo{u^{n,h_r}(s)}-\sigma_k\kuo{u^{n,h}(s)}}\dot{h}(s)ds\\
  	&+ 2 \sum_{k=1}^m \int_0^t \brackt{u^{n,h_r}(s),\sigma_k\kuo{u^{n,h_r}(s)}}\kuo{\dot{h}_r(s)-\dot{h}(s)}ds  \\
  	&+2 \int_0^t\brackt{\skpen(s)-u^{m,h}(s),b\kuo{\skpen(s)}-b\kuo{u^{m,h}(s)}}ds\\
  	& -2n\int_0^t\brackt{\skpen(s)-u^{m,h}(s),\gamma\kuo{\skpen(s)}\abs{\skpen(s) - \pi\kuo{\skpen(s)}}}ds\\
  	&+2m\int_0^t\brackt{\skpen(s)-u^{m,h}(s),\gamma\kuo{u^{m,h}(s)}\abs{u^{m,h}(s)-\pi\kuo{u^{m,h}(s)}}}ds
  	  \end{aligned}
\end{equation}

Since $h_r$ weak converges to $h$ as $r\rightarrow 0$, we have 
\begin{equation}
  \lim\limits_{r\rightarrow 0}2 \sum_{k=1}^m \int_0^t \brackt{u^{n,h_r}(s),\sigma_k\kuo{u^{n,h_r}(s)}}\kuo{\dot{h}_r(s)-\dot{h}(s)}ds  = 0
\end{equation}
By a similar argument as Lemma \ref{lem:6}, we can prove Lemma \ref{lem:7}.
\end{proof}
\begin{proposition}\label{pro:1}
For $N>1$, $h_{r}$ weak convergence to $h$ as $r \rightarrow 0$, then for any $T>0$, $n\geq 1$, $\kuo{u^{h_r},\eta^{h_r}}$ converges to $\kuo{u^h,\eta^h}$.
\end{proposition}
\begin{proof}
	We only need to show 
		\begin{equation}
\lim _{r \rightarrow 0}\left(\sup _{0 \leq t \leq T}\left\|u^{n, h_{r}}(t)-u^{n, h}(t)\right\|_H^2+\int_0^T\left\|u^{n, h_{r}}(t)-u^{n, h}(t)\right\|_V^2 d t\right)=0 
\end{equation} 
We have 
\begin{equation}
  \begin{aligned}
&\left\|u^{n, h_{r}}(t)-u^{n, h}(t)\right\|_H^2+\int_0^T\left\|u^{n, h^{(r)}}(t)-u^{n, h}(t)\right\|_V^2 d t \\
&\leq 3\norm{u^{n,h_r}-u^{h_r}}_H^2 +3\norm{u^{n,h_r}-u^{h}}_H^2+3\norm{u^{n,h}-u^{n}}_H^2\\
&+ 6\int_0^T\norm{u^{n,h_r}-u^{h_r}}_V^2dt + 6\int_0^T\norm{u^{n,h_r}-u^{h}}_V^2dt + \int_0^T\norm{u^{n,h}-u^{n}}_H^2dt
 \end{aligned}
\end{equation}
From Lemma \ref{lem:6} , for any $\alpha$, there exists $N>0$ for any $n\geq N$,
\begin{equation}\label{eq:pro_1}
  \sup_{0\leq t\leq T}\norm{u^{n,h_r}-u^{h_r}}^2_H + \int_0^T\norm{u^{n,h_r}-u^{h_r}}^2_Vdt \leq\frac{1}{18}\alpha
\end{equation}
\begin{equation}\label{eq:pro_2}
  \sup_{0\leq t\leq T}\norm{u^{n,h}-u^{h}}^2_H + \int_0^T\norm{u^{n,h}-u^{h}}^2_Vdt \leq\frac{1}{18}\alpha
\end{equation}
From Lemma \ref{lem:7}, there exists $r_0>0$, for any $r\in \kuo{0,r_0}$, we have 
\begin{equation}\label{eq:pro_3}
  \sup_{0\leq t\leq T}\norm{u^{n,h_r}-u^{n,h}}^2_H + \int_0^T\norm{u^{n,h_r}-u^{n,h}}^2_Vdt \leq\frac{1}{18}\alpha
\end{equation}

Combining \eqref{eq:pro_1}, \eqref{eq:pro_2} and \eqref{eq:pro_3}, we can show 
$$
\lim _{r \rightarrow 0}\left(\sup _{0 \leq t \leq T}\left\|u^{n, h_{r}}(t)-u^{n, h}(t)\right\|_H^2+\int_0^T\left\|u^{n, h_{r}}(t)-u^{n, h}(t)\right\|_V^2 d t\right)=0 
$$
\end{proof}
\section{Large deviation principle}\label{sec:4}
For any $\varepsilon >0$, we can define a measurable mapping $\mathcal{G}^\varepsilon:C([0,T];H)\rightarrow C([0,T];H)\cap L^2([0,T];V)$
\begin{equation}
  \mathcal{G}^\varepsilon(B(\cdot)) := u^\varepsilon
\end{equation}
where $u^\varepsilon$ is the solution of \eqref{eq:sn_SPDE}. Let $\dkuo{h^\varepsilon}_{\varepsilon>0}\in \mathcal{A}_N$, from Girsanov Theorem, $Y^\varepsilon:= \mathcal{G}^\varepsilon\kuo{B(\cdot)+\frac{1}{\sqrt{\varepsilon}}\int_0^{\cdot}\dot{h}^\varepsilon(s)ds}$ is the solution of \eqref{eq:c_SPDE} 
\begin{equation}\label{eq:c_SPDE}
\left\{\begin{aligned}
&dY_i^\varepsilon(t,x) = \frac{\partial^2 Y^\varepsilon(t,x)}{\partial x^2}dt + b_i\kuo{\contro(t,x)}dt + \sqrt{\varepsilon}\sum_{k=1}^m \sigma_k\kuo{\contro(t,x)}dB_j
\\
& \quad \quad \quad \quad + \sum_{k=1}^m \sigma_k\kuo{\contro(t,x)}h^\varepsilon_k(t)dt -\gamma_i\kuo{\contro(t,x)}dk_i(t,x)\quad x \in [0,1]\\
& Y^\varepsilon(0,\cdot) = (Y^\varepsilon_1(0,\cdot), Y^\varepsilon_2(0,\cdot),\cdots,Y^\varepsilon_n(0,\cdot))^T \in \bar{\mathcal{O}}\\
& Y^\varepsilon(t,0) = Y(t,1) = 0
\end{aligned}\right.
\end{equation}

Let $Z = \mathcal{G}^0\kuo{\int_0^\cdot \dot{h}^\varepsilon(s)ds}$ is the solution of \eqref{eq:Z}.
\begin{equation}\label{eq:Z}
\left\{\begin{aligned}
&d Z^{\varepsilon}(t, x)= \frac{\partial^2 Z^{\varepsilon}(t, x)}{\partial x^2} d t+b_i\kuo{Z^\varepsilon(t,x)}dt+\sum_{j=1}^m \sigma_j\left(Z^{\varepsilon}(t, x)\right) \dot{h}_j^{\varepsilon}(t) d t\\
&\quad\quad\quad \quad -\gamma_i\kuo{Z^\varepsilon(t,x)} d k_i^{\varepsilon, Z}(t, x), \quad x \in[0,1] \\
&Z^{\varepsilon}(0, \cdot) = \kuo{Z^\varepsilon_1(0,\cdot),Z^\varepsilon_2(0,\cdot),\cdots,Z^\varepsilon_n(0,\cdot)}^T\in \bar{\mathcal{O}} \\
&Z^{\varepsilon}(t, 0)= Z^{\varepsilon}(t, 1)=0
\end{aligned}\right.
\end{equation}
We using penalized method to get some priori estimates. Consider penalized system of $Y^\varepsilon$ and $Z^\varepsilon$.
\begin{equation}\label{eq:pen_Y}
  \begin{aligned}
  \pencontro(t,x) &= u(0,x) + \int_0^t \frac{\partial^2 \pencontro (s,x)}{\partial x^2}ds + \int_0^t\sigma\kuo{\pencontro(s,x)}\dot{h}^\varepsilon(t)dt + \int_0^tb\kuo{\pencontro(s,x)}ds\\
  &+\int_0^t\sigma\kuo{\pencontro(s,x)}\dot{h}^\varepsilon(s)ds +\sqrt{\varepsilon}\int_0^t\sigma\kuo{\pencontro(s,x)}dB_s \\
  &- n\int_0^t\gamma\kuo{\pencontro(s,x)}\abs{\pencontro(s,x)-\pi\kuo{\pencontro(s,x)}}ds
  \end{aligned}
\end{equation}
\begin{equation}\label{eq:pen_Z}
  \begin{aligned}
  \penhsk(t,x) &= u(0,x) + \int_0^t\frac{\partial^2 \penhsk(s,x)}{\partial x^2}ds + \int_0^tb_i\kuo{\penhsk(s,x)}ds + \int_0^t\sigma\kuo{\penhsk(s,x)}\dot{h}^\varepsilon(s)ds\\
  &-n\int_0^t\gamma\kuo{\penhsk(s,x)}\abs{\penhsk(s,x) - \pi\kuo{\penhsk(s,x)}}ds
  \end{aligned}
\end{equation}

As $n\rightarrow \infty$, we have $\pencontro \rightarrow Y^\varepsilon$ and $\penhsk \rightarrow Z^\varepsilon$.
We need some priori estimates for $\pencontro$ and $\penhsk$. Those proof are similar to \cite{Duan2019WhiteND}, we sketch it here.
\begin{lemma}\label{lem:8}
	Let $\nu^n = \dkuo{\pencontro, \penhsk}$
	\begin{enumerate}[(i)]
		\item  we have constant $M_T$ such that
		\begin{equation}
  \sup_n\e\zkuo{\kuo{n\int_0^T \norm{\nu^n(t)-\pi\kuo{\nu^n}}_{L^1([0,1])}dt}^2}\leq M_T
\end{equation}
		\item 
	\begin{equation}
  \lim\limits_{n\rightarrow \infty}\e\zkuo{\sup_{0\leq t\leq T}\norm{\nu^n(t)-\pi\kuo{\nu^n(t)}}_H^2}=0
  \end{equation}
  \item 
  \begin{equation}
  \lim\limits_{n\rightarrow \infty} \e\zkuo{\sup_{0\leq t\leq T}\norm{\nu^n(t)-\pi\kuo{\nu^n(t)}}_{L^\infty([0,1])}^2}=0
\end{equation}
	\item 
	Assume that $\nu^n(0) \in V$. Then we have 
	\begin{equation}
  \begin{aligned}
  	\sup_{n}\e\zkuo{\sup_{0\leq t \leq T}\norm{\nu^n(t)}^2_V}&\leq \infty\\
  	\sup_{n}\e\zkuo{\int_0^T\norm{\nu^n(t)}^2_{H^2}}&\leq \infty
  \end{aligned}
\end{equation}

	\end{enumerate}	
\end{lemma}
Based on penalized system of  \eqref{eq:pen_Y} and \eqref{eq:pen_Z}, mimicking the proof of \cite{Duan2019WhiteND}[Theorem 3.9] , we have the following Lemma.
\begin{lemma}\label{lem:LDP_1}
	For any $\dkuo{h^\varepsilon}\in \mathcal{A}_N$
	\begin{equation}
\begin{aligned}
&\mathbb{E}\left[\sup _{0 \leq t \leq T}\left\|Y^{\varepsilon}(t)\right\|_H^2+\int_0^T\left\|Y^{\varepsilon}(t)\right\|_V^2 d t\right]<\infty, \\
&\mathbb{E}\left[\sup _{0 \leq t \leq T}\left\|Z^{\varepsilon}(t)\right\|_H^2+\int_0^T\left\|Z^{\varepsilon}(t)\right\|_V^2 d t\right]<\infty .
\end{aligned}
\end{equation}
\end{lemma}

\begin{lemma}\label{lem:LDP_2}
	Existing $\lambda>0$ such that
	\begin{equation}
\begin{aligned}
\lim _{\varepsilon \rightarrow 0}\{\mathbb{E}& {\left[\sup _{0 \leq t \leq T}\left(\exp \left\{-\lambda \int_0^t\left(\left\|Y^{\varepsilon}(s)\right\|_V^2+\left\|Z^{\varepsilon}(s)\right\|_V^2\right) d s\right\}\left\|Y^{\varepsilon}(t)-Z^{\varepsilon}(t)\right\|_H^2\right)\right] } \\
&\left.+\mathbb{E}\left[\int_0^T \exp \left\{-\lambda \int_0^t\left(\left\|Y^{\varepsilon}(s)\right\|_V^2+\left\|Z^{\varepsilon}(s)\right\|_V^2\right) d s\right\}\left\|Y^{\varepsilon}(t)-Z^{\varepsilon}(t)\right\|_V^2 d t\right]\right\}=0 .
\end{aligned}
\end{equation} 
\end{lemma}
\begin{proof}
	We only need to show 
	\begin{equation}
  \begin{aligned}
\lim_{\varepsilon \rightarrow 0}\lim_{n\rightarrow \infty}\{\mathbb{E}& {\left[\sup _{0 \leq t \leq T}\left(\exp \left\{-\lambda \int_0^t\left(\left\|Y^{n,\varepsilon}(s)\right\|_V^2+\left\|Z^{n,\varepsilon}(s)\right\|_V^2\right) d s\right\}\left\|Y^{n,\varepsilon}(t)-Z^{n,\varepsilon}(t)\right\|_H^2\right)\right] } \\
&\left.+\mathbb{E}\left[\int_0^T \exp \left\{-\lambda \int_0^t\left(\left\|Y^{n,\varepsilon}(s)\right\|_V^2+\left\|Z^{n,\varepsilon}(s)\right\|_V^2\right) d s\right\}\left\|Y^{n,\varepsilon}(t)-Z^{n,\varepsilon}(t)\right\|_V^2 d t\right]\right\}=0 .
\end{aligned}
\end{equation}

Define 
$$
\psi^n(t) = \exp\dkuo{-\lambda\int_0^t\kuo{\norm{\pencontro(s)}^2_V + \norm{\penhsk(s)}^2_V}ds}
$$
Applying Ito's formula, we have 
\begin{equation}
  \begin{aligned}
  	\norm{\pencontro(t)-\penhsk(t)}^2_H\psi^n(t) &= -\lambda\int_0^t\psi^n\kuo{\norm{\pencontro(s)}^2_V+\norm{\penhsk(s)}^2_V}\norm{\pencontro(s) - \penhsk(s)}^2_Hds\\
  	&-2\int_0^t\psi^n\norm{\pencontro(s)-\penhsk(s)}^2_Vds\\
  	&+2\int_0^t\psi^n\brackt{\pencontro(s)-\penhsk(s),b\kuo{\pencontro(s)}-b\kuo{\penhsk(s)}}ds\\
  	&+2\sqrt{\varepsilon}\sum_{k=1}^m\int_0^t\psi^n\brackt{\pencontro(s)-\penhsk(s),\sigma_j\kuo{\pencontro(s)}}dB_j(s)\\
  	&+\varepsilon\sum_{k=1}^m\int_0^t\psi^n\norm{\sigma_j\kuo{\pencontro(s)}}^2_Hds\\
  	&+2\int_0^t\psi^n\brackt{\pencontro(s)-\penhsk(s),\zkuo{\sigma\kuo{\pencontro(s)}-\sigma\kuo{\penhsk(s)}}\dot{h}^\varepsilon(s)}ds\\
  	&-2n\int_0^t\psi^n\brackt{\pencontro(s)-\penhsk(s),\gamma\kuo{\pencontro(s,x)}\abs{\pencontro(s,x)-\pi\kuo{\pencontro(s,x)}}}ds\\
  	&+2n\int_0^t\psi^n\brackt{\pencontro(s)-\penhsk(s),\gamma\kuo{\penhsk(s,x)}\abs{\penhsk(s,x)-\pi\kuo{\penhsk(s,x)}}}ds\\
  	&=I_1 + I_2+I_3+I_4+I_5+I_6+I_7+I_8
  \end{aligned}
\end{equation}
\begin{equation}
  \begin{aligned}
  	I_3\leq \frac{1}{2}\sup_{0\leq t\leq T}\kuo{\psi^n(s)\norm{\pencontro(s) - \penhsk(s)}^2_H}+ C \int_0^T\psi^n(s)\norm{\pencontro(s)-\penhsk(s)}_H^2ds
  \end{aligned}
\end{equation}

Apply B-D-G inequality, we have 
\begin{equation}\label{eq:I_4}
  \begin{aligned}
  	\e\zkuo{\sup_{0\leq t\leq T}I_4}&\leq 2\sqrt{ \varepsilon}C \e\dkuo{\zkuo{\sum_{k=1}^m\int_0^T\psi^n(s)^2\brackt{\pencontro(s)-\penhsk(s),\sigma_j\kuo{\pencontro(s)}}^2ds}^{\frac{1}{2}}}\\
  	&\leq 2\sqrt{\varepsilon}C_1 \e\dkuo{\sup_{0\leq t\leq T}\psi^n(s)^{\frac{1}{2}}\norm{\pencontro(s)-\penhsk(s)}_H\kuo{\sum_{k=1}^m\int_0^T\psi^n(s)\norm{\sigma_j\kuo{\pencontro(s)}}^2_Hds}^{\frac{1}{2}}}\\
  	&\leq \frac{1}{4}\e\zkuo{\sup_{0\leq t\leq T}\psi^n(s)\norm{\pencontro(s)-\penhsk(s)}^2_H}+ 16\varepsilon C_1^2 \e\zkuo{\int_0^T\psi^n(s)\kuo{1+\norm{\pencontro(s)}_H^2}ds}
  \end{aligned}
\end{equation}
\begin{equation}\label{eq:I_5}
  I_5\leq \varepsilon C\int_0^T\psi^n(s)\kuo{1+\norm{\pencontro(s)}^2_H}ds
\end{equation}
\begin{equation}\label{eq:I_6}
  I_6\leq \frac{1}{2}\sup_{0\leq t\leq T}\kuo{\psi^n(s)\norm{\pencontro(s) - \penhsk(s)}^2_H}+ C \int_0^T\psi^n(s)\norm{\pencontro(s)-\penhsk(s)}_H^2ds
\end{equation}
\begin{equation}\label{eq:I_7}
  \begin{aligned}
  	I_7&=-2n\int_0^t\brackt{\psi^n(s)\kuo{\pencontro(s) - \pi\kuo{\penhsk(s)}},\gamma\kuo{\pencontro(s)}\abs{\pencontro(s)-\pi\kuo{\pencontro(s)}}}ds\\
  	&+2n\int_0^t\brackt{\psi^n(s)\kuo{\penhsk(s)-\pi\kuo{\penhsk(s)}},\gamma\kuo{\pencontro(s)}\abs{\pencontro(s)-\pi\kuo{\pencontro(s)}}}ds\\
  	&\leq C\kuo{\int_0^t\psi^n(s)\norm{\pencontro(s) - \penhsk(s)}^2_Hds}^{\frac{1}{2}}\kuo{n^2\int_0^t\psi^n(s)\norm{\pencontro(s)-\pi\kuo{\pencontro(s)}}^2_Hds}^{\frac{1}{2}}\\
  	&+C\sup_{0\leq s\leq t}\psi^n(s)\norm{\penhsk(s)-\pi\kuo{\penhsk(s)}}_H\times\kuo{n\int_0^t\psi^n(s)\norm{\pencontro(s)-\pi\kuo{\pencontro(s)}}_Hds}\\
  	&+ \sup_{0\leq s\leq t}\psi^n(s)\norm{\penhsk(s)-\pi\kuo{\penhsk(s)}}_{L^\infty([0,1])}\times \kuo{2n \int_0^t\psi^n(s)\norm{\pencontro(s)-\pi\kuo{\pencontro(s)}}_{L^1([0,1])}ds}
  \end{aligned}
\end{equation}

By a similar argument as above ,the same estimate holds for $I_8$.
Combining the estimates for $I_3$, $I_4$, $I_5$, $I_6$, $I_7$ and $I_8$ and taking expectation, we can deduce that 
\begin{equation}
  \begin{aligned}
  	&\e\zkuo{\sup_{0\leq t\leq T}\psi^n(t)\norm{\pencontro(s)-\penhsk(s)}^2_H}+ \e\zkuo{\int_0^T\psi^n(s)\norm{\pencontro(s)-\penhsk(s)}^2_Vds}\\
  	&\leq \frac{5}{4}\e\zkuo{\sup_{0\leq t\leq T}\psi^n(t)\norm{\pencontro(t)-\penhsk(t)}^2_H}+16\varepsilon C_1^2\e\zkuo{\int_0^T\psi^n(s)\kuo{1+\norm{\pencontro(s)}}^2_Hds}\\
  	&+\varepsilon C\e\zkuo{\int_0^T\psi^n(s)\kuo{1+\norm{\pencontro(s)}^2_H}ds} + \frac{1}{2}\e\zkuo{\sup_{0\leq t\leq T}\kuo{\psi^n(s)\norm{\pencontro(s)-\penhsk(s)}^2_H}}\\
  	&+C\e\zkuo{\int_0^T\psi^n(s)\norm{\pencontro(s)-\penhsk(s)}_H^2ds}+C_2\left\{\e\zkuo
  	\int_0^T\psi^n(s)\norm{\pencontro(s)-\penhsk(s)}^2_Hds\right.\\
  	&\left.+\e\zkuo{\sup_{0\leq t\leq T}\psi^n(t)\norm{\penhsk(t)-\pi\kuo{\penhsk(t)}}_H^2}+\e\zkuo{\sup_{0\leq t \leq T }\norm{\pencontro(s)-\pi\kuo{\pencontro(s)}}^2_H}\right\}^{\frac{1}{2}}\\
  	&+C_3\dkuo{\e\zkuo{\sup_{0\leq t\leq T}\psi^n(s)\norm{\penhsk(s)-\pi\kuo{\penhsk(s)}}^2_{L^{\infty}([0,1])}}}^{\frac{1}{2}}\\
  	&\times \dkuo{\e\zkuo{\kuo{2n\int_0^T\psi^n(s)\norm{\penhsk(s)-\pi\kuo{\penhsk(s)}}_{L^1([0,1])}}}^2}^{\frac{1}{2}}
  \end{aligned}
\end{equation}

By Gronwall's inequality, $\psi^n(s)\leq 1$ as $\lambda$ enough large, and apply Lemma \ref{lem:8}, we have 

\begin{equation}
  \begin{aligned}
  	&\e\zkuo{\sup_{0\leq t\leq T}\psi^n(t)\norm{\pencontro(s)-\penhsk(s)}^2_H}+ \e\zkuo{\int_0^T\psi^n(s)\norm{\pencontro(s)-\penhsk(s)}^2_Vds}\\
  	&\leq C_4\dkuo{\e\zkuo{\sup_{0\leq t\leq T}\norm{\penhsk(t)-\pi\kuo{\penhsk(t)}}_{L^\infty([0,1])}^2}}^{\frac{1}{2}}\\
  	&+C_5\dkuo{\e\zkuo{\sup_{0\leq t\leq T}\norm{\pencontro(s)-\pi\kuo{\pencontro(s)}}^2_{L^\infty([0,1])}}}^{\frac{1}{2}}
  \end{aligned}
\end{equation}
Then we have 

	\begin{equation}
  \begin{aligned}
\lim_{\varepsilon \rightarrow 0}\lim_{n\rightarrow \infty}\{\mathbb{E}& {\left[\sup _{0 \leq t \leq T}\left(\exp \left\{-\lambda \int_0^t\left(\left\|Y^{n,\varepsilon}(s)\right\|_V^2+\left\|Z^{n,\varepsilon}(s)\right\|_V^2\right) d s\right\}\left\|Y^{n,\varepsilon}(t)-Z^{n,\varepsilon}(t)\right\|_H^2\right)\right] } \\
&\left.+\mathbb{E}\left[\int_0^T \exp \left\{-\lambda \int_0^t\left(\left\|Y^{n,\varepsilon}(s)\right\|_V^2+\left\|Z^{n,\varepsilon}(s)\right\|_V^2\right) d s\right\}\left\|Y^{n,\varepsilon}(t)-Z^{n,\varepsilon}(t)\right\|_V^2 d t\right]\right\}=0 .
\end{aligned}
\end{equation}
\end{proof}

\begin{theorem}
	$u^\varepsilon$ be the unique solution to \eqref{eq:sn_SPDE}. Then the family of $\dkuo{u^\varepsilon}_{\varepsilon>0}$ satisfies a LDP on the space $\mathcal{E}$ with rare function 
	$$
	   I(g) = \inf_{\dkuo{h\in \mathcal{H},g = \mathcal{G}^0\kuo{\int_0^\cdot \dot{h}(s)ds}}}\dkuo{\frac{1}{2}\norm{h}^2_{\mathcal{H}}}
	$$
\end{theorem}
	\begin{proof}
According to Theorem \ref{thm:LDP_condition}, to complete the proof of this theorem, it is sufficient to verify \textbf{(LDP1)} and \textbf{(LDP2)}.

From Proposition \ref{pro:1}, we have established the \textbf{(LDP2)}. Then we only to prove \textbf{(LDP1)} is true.

Let $\psi = \lim_{n\rightarrow \infty}\psi^n$.
	\begin{equation}\label{eq:LDP_1}
\begin{aligned}
& \mathbb{P}\left(\sup _{0 \leq t \leq T}\left\|Y^{\varepsilon}(t)-Z^{\varepsilon}(t)\right\|_H^2+\int_0^T\left\|Y^{\varepsilon}(t)-Z^{\varepsilon}(t)\right\|_V^2 d t>\delta\right) \\
=& \mathbb{P}\left(\sup _{0 \leq t \leq T}\left\|Y^{\varepsilon}(t)-Z^{\varepsilon}(t)\right\|_H^2+\int_0^T\left\|Y^{\varepsilon}(t)-Z^{\varepsilon}(t)\right\|_V^2 d t>\delta, \int_0^T\left(\left\|Y^{\varepsilon}(s)\right\|_V^2+\left\|Z^{\varepsilon}(s)\right\|_V^2\right) d s>M\right) \\
&+\mathbb{P}\left(\sup _{0 \leq t \leq T}\left\|Y^{\varepsilon}(t)-Z^{\varepsilon}(t)\right\|_H^2+\int_0^T\left\|Y^{\varepsilon}(t)-Z^{\varepsilon}(t)\right\|_V^2 d t>\delta, \int_0^T\left(\left\|Y^{\varepsilon}(s)\right\|_V^2+\left\|Z^{\varepsilon}(s)\right\|_V^2\right) d s \leq M\right) \\
\leq & \mathbb{P}\left(\int_0^T\left(\left\|Y^{\varepsilon}(s)\right\|_V^2+\left\|Z^{\varepsilon}(s)\right\|_V^2\right) d s>M\right) \\
&+\mathbb{P}\left(\sup _{0 \leq t \leq T} \exp ^2\left\{-\lambda \int_0^t\left(\left\|Y^{\varepsilon}(s)\right\|_V^2+\left\|Z^{\varepsilon}(s)\right\|_V^2\right) d s\right\}\left\|Y^{\varepsilon}(t)-Z^{\varepsilon}(t)\right\|_H^2\right.\\
&\left.+\int_0^T \exp \left\{-\lambda \int_0^t\left(\left\|Y^{\varepsilon}(s)\right\|_V^2+\left\|Z^{\varepsilon}(s)\right\|_V^2\right) d s\right\}\left\|Y^{\varepsilon}(t)-Z^{\varepsilon}(t)\right\|_V^2 d t \geq e^{-\lambda M} \delta\right) \\
\leq & \frac{1}{M} \mathbb{E}\left[\int_0^T\left(\left\|Y^{\varepsilon}(s)\right\|_V^2+\left\|Z^{\varepsilon}(s)\right\|_V^2\right) d s\right] \\
&+\frac{e^{\lambda M}}{\delta} \mathbb{E}\left[\sup _{0 \leq t \leq T} \psi(t)\left\|Y^{\varepsilon}(t)-Z^{\varepsilon}(t)\right\|_H^2+\int_0^T E_{Y^{\varepsilon}, Z^{\varepsilon}}(t)\left\|Y^{\varepsilon}(t)-Z^{\varepsilon}(t)\right\|_V^2 d t\right] .
\end{aligned}
\end{equation}

From Lemma \ref{lem:LDP_1}, for any $\alpha >0$, exists $M>0$ such that 
\begin{equation}\label{eq:LDP_2}
  \frac{1}{M}\e\zkuo{\int_0^T\kuo{\norm{Y^\varepsilon(s)}^2_V + \norm{Z^\varepsilon(s)}^2_V}ds}\leq \frac{\alpha}{2}
\end{equation}

From Lemma \ref{lem:LDP_2}, there exists $\varepsilon_0$, for any $\varepsilon\in (0,\varepsilon_0)$ we have 
\begin{equation}\label{eq:LDP_3}
\frac{e^{\lambda M}}{\delta} \mathbb{E}\left[\sup _{0 \leq t \leq T} \psi(t)\left\|Y^{\varepsilon}(t)-Z^{\varepsilon}(t)\right\|_H^2+\int_0^T \psi(t)(t)\left\|Y^{\varepsilon}(t)-Z^{\varepsilon}(t)\right\|_V^2 d t\right] \leq \frac{\alpha}{2} .
\end{equation}

Combing \eqref{eq:LDP_1}, \eqref{eq:LDP_2} and \eqref{eq:LDP_3}, we have for any $\varepsilon \in \kuo{0,\varepsilon_0}$
\begin{equation}
\mathbb{P}\left(\sup _{0 \leq t \leq T}\left\|Y^{\varepsilon}(t)-Z^{\varepsilon}(t)\right\|_H^2+\int_0^T\left\|Y^{\varepsilon}(t)-Z^{\varepsilon}(t)\right\|_V^2 d t>\delta\right) \leq \alpha
\end{equation}

Therefore we can easily show that
\begin{equation}
\lim_{\varepsilon \rightarrow 0}\mathbb{P}\left(\sup _{0 \leq t \leq T}\left\|Y^{\varepsilon}(t)-Z^{\varepsilon}(t)\right\|_H^2+\int_0^T\left\|Y^{\varepsilon}(t)-Z^{\varepsilon}(t)\right\|_V^2 d t>\delta\right) =0.
\end{equation}
\end{proof}

\bibliographystyle{plain}
\bibliography{ref}
\end{document}